\documentclass[a4paper, 11pt]{amsart}
\usepackage[utf8]{inputenc}
\usepackage[T1]{fontenc}
\usepackage{lmodern}

\usepackage[utf8]{inputenc}
\usepackage{float}
\usepackage[english]{babel}
\usepackage{geometry} 

\usepackage[shortlabels]{enumitem}
\setlist[enumerate]{label=\rm{(\arabic*)}}
\setlist[enumerate,2]{label=\rm{(\roman*)}}
\setlist[itemize]{label=\raisebox{0.25ex}{\tiny$\bullet$}}

\usepackage{amsmath,amssymb}
\usepackage{adjustbox}
\usepackage{booktabs} 
\usepackage{color}
\usepackage{amsthm}
\usepackage{graphicx}
\usepackage{listings}

\usepackage{emptypage}

\usepackage{mathrsfs}
\usepackage{mathtools}
\usepackage[all,2cell]{xy}
\UseAllTwocells

\usepackage{ stmaryrd }

\usepackage{hyperref}
\hypersetup{
   colorlinks=true,
    linkcolor=teal,
    citecolor=brown,
    filecolor=pink,
    urlcolor=brown,
    linktocpage=true
}
\usepackage[capitalize]{cleveref}

\theoremstyle{plain}

\newtheorem{theorem}[subsection]{Theorem}
\newtheorem*{theorem*}{Theorem}
\newtheorem{lemma}[subsection]{Lemma}
\newtheorem*{lemma*}{Lemma}
\newtheorem{proposition}[subsection]{Proposition}
\newtheorem*{proposition*}{Proposition}
\newtheorem{corollary}[subsection]{Corollary}
\newtheorem{proposition-definition}[subsection]{Proposition-Definition}
\theoremstyle{definition}

\newtheorem{defn}[subsection]{Definition}
\newtheorem{notation}[subsection]{Notation}
\theoremstyle{remark} 

\newcommand{\ie}{\emph{i.e.}\ }

\numberwithin{equation}{section}

\newtheorem{decomposition-theorem}[subsection]{Decomposition Theorem}

\newtheorem{example}[subsection]{Example}
\newtheorem{remark}[subsection]{Remark}
\newtheorem{construction}[subsection]{Construction}
\newtheorem{warning}[subsection]{Warning}
\newcommand{\reminder}[1]{}

\newcommand{\Mod}{\mathrm{Mod}\,}

\newcommand{\op}{^{\mathrm{op}}}

\newcommand{\essIm}{\mathrm{essIm}}
\newcommand{\cCat}{\mathrm{Cat}}
\newcommand{\cCAT}{\mathrm{CAT}}
\newcommand{\cPDER}{\mathcal{PD}\mathit{ER}}
\newcommand{\cDER}{\mathcal{D}\mathit{ER}}

\newcommand{\Ho}{\mathrm{Ho}}

\newcommand{\colim}{\mathrm{colim}}

\newcommand{\Z}{\mathbb{Z}}
\newcommand{\N}{\mathbb{N}}

\newcommand{\id}{\mathrm{id}}
\newcommand{\mL}{\mathrm{L}}

\newcommand{\cc}{{\mathcal C}}

\newcommand{\D}{\mathscr{D}}
\newcommand{\E}{\mathscr{E}}

\renewcommand{\phi}{\varphi}

\renewcommand{\hat}[1]{\widehat{#1}}

\renewcommand{\tilde}[1]{\widetilde{#1}}

\newcommand{\cSp}{\mathcal{S}\!\mathit{p}}

\usepackage{setspace}

\def\oneline{\vskip12pt}

\usepackage[utf8]{inputenc}

\usepackage{graphicx}
\usepackage{tikz-cd}

\usepackage[all]{xy}
\usepackage{epsf}

\usepackage{graphicx,amssymb}

\title{$\infty$-Dold-Kan correspondence via representation theory}
\author{Chiara Sava}
\address[Chiara Sava]{Univerzita Karlova, Matematicko-fyzik\'{a}ln\'{i} fakulta, Katedra Algebry, Sokolovsk\'{a} 49/83, 186 75 Praha 8, \v{C}esk\'a republika.}
\email{sava@karlin.mff.cuni.cz}
\subjclass[2020]{55U35 (Primary), 16G20, 16E35, 18G80 (Secondary)}
\keywords{Derivators, Representation Theory, Tilting Theory, Dold-Kan correspondence}
\thanks{Corresponding Author: Chiara Sava (sava@karlin.mff.cuni.cz).}

\begin{document}

\begin{abstract}
We give a purely derivator-theoretical reformulation and proof of a classic result of Happel and Ladkani, showing that it occurs uniformly across stable derivators and it is then independent of coefficients. The resulting equivalence provides a bridge between homotopy theory and representation theory: indeed, our result is a derivator-theoretic version of the $\infty$-Dold-Kan correspondence for bounded chain complexes. Moreover, our equivalence can also be realized as an action of a spectral bimodule in the setting of universal tilting theory developed by Groth and {\v S}{\v t}ov{\' i}{\v c}ek.  
\end{abstract}

\maketitle
\tableofcontents

\section{Introduction}
Over algebraically closed fields, one of the best understood examples of finite dimensional algebras are path algebras over Dynkin quivers \cite{gabriel} and their quotients by admissible ideals \cite{assemsimsonskowronski06}. The most relevant theory in this setting is due to Auslander and Reiten \cite{auslanderreiten75, auslanderreiten77}. In particular, this theory motivates why, in order to understand algebras, we study not only the category of finitely generated modules but also the associated bounded derived category. At this level, equivalences are usually found by applying the derived version of the Morita theory due to Rickard \cite{rickard89}. This theory, which is meant as a generalization of the tilting theory developed by Happel and Ringel \cite{happelringel82}, is based on the study of the so called \textit{tilting complex} (see also \cite{happel87}). Then, one obtains the desired equivalences as derived tensor products by tilting complexes. 

\medskip

Let $A_n$ be the Dynkin quiver of type $A$ with $n$ vertices. For a commutative ring $k$, Ladkani \cite{lad:rectangles} studied the construction of new tilting complexes realizing derived equivalences between tensor products of $k$-algebras over Dynkin quivers of type $A$ and $k$-algebras over the $A_n$ quiver with relations

\[ A_n \colon
\begin{tikzcd}
0 \arrow[rr, "\alpha_{0}"] &  & 1 \arrow[rr, "\alpha_{1}"] &  & 2 \arrow[rr, "\alpha_2"] &  & \cdots \arrow[rr, "\alpha_{n-2}"] &  & n-1
\end{tikzcd}.\]
A trivial case of these equivalences, which is also a consequence of the work by Happel \cite{happel88}, can be stated as follows. 
\begin{theorem}[{\cite[Corollary 1.2]{lad:rectangles}}]\label{dercat}Let $I$ be the ideal of the path algebra $kA_n$ generated by the relations $\alpha_{i+1}\alpha_i=0$ for $0 \leq i \leq n-3$. Then, there is an equivalence of derived categories of modules
\begin{equation}\label{derivedeq}
    \mathrm{D}(kA_n/I) \cong \mathrm{D}(kA_n)\,.\end{equation}
    \end{theorem}
We aim to enhance and generalize the equivalence (\ref{derivedeq}) and for this purpose the language of derivators (see \cite{derptderstder}) turns out to be the most convenient one. Derivators, in fact, are meant to be a minimal extension of a derived category with a well behaved calculus of homotopy limits and colimits. The easiest example of a derivator consists of diagrams in a bicomplete category, where the left and right Kan extensions can always be computed pointwise (see \cite{maclane}). The definition of derivators axiomatises this property which, for derived categories in \cref{dercat}, is guaranteed if we consider the derived category of diagram categories (\emph{coherent diagrams}) instead of diagrams in the derived category (\emph{incoherent diagrams}).

\medskip

Specifically, in this article we work with stable derivators which, by definition, admit zero objects and whose homotopy pushout squares and homotopy pullback squares coincide. Stable derivators, introduced by Heller \cite{heller88, heller97} and Grothendieck \cite{grothendieck91}, were then studied further by Franke \cite{franke96}, Keller \cite{keller91} and Maltsiniotis \cite{maltsiniotis01, maltsiniotis07}. They are of general interest because, with the additional hypothesis of being \textit{strong} (cf. \cite[Definition 1.8]{derptderstder}), the underlying category of a strong and stable derivator is always a triangulated category. They are interesting for us because, given a Grothendieck category $\mathcal{G}$, the derived category of diagrams in $\mathcal{G}$ forms a so-called stable derivator. 

\medskip

In this article, we give a purely derivator-theoretic reformulation and proof of  Theorem \ref{dercat}; this shows that the phenomenon occurs uniformly across stable derivators and is independent of the ring of coefficients $k$. Namely, given a stable derivator $\D$, we refine the derived category $\mathrm{D}(kA_n)$ with the stable derivator shifted by the free category of $A_n$ (\cref{shifted}), which we denote by $\D^{A_n}$. The enhancement of the derived category $\mathrm{D}(kA_n/I)$ is more complicated to define as it involves the relations given by the ideal $I$. Indeed, we introduce the new notion of \textbf{strict full subderivator} (Section \ref{strictfullsubder}) which allows us to express the relations in the language of derivators and to define the correct enhancement as a particular strict full subderivator \begin{align}
\label{derivator intro}
\D^{A(n,2)} \subseteq \D^{\tilde{A}(n,2)}
\end{align} of the shifted derivator \[
\D^{\tilde{A}(n,2)}\]
where $\tilde{A}(n,2)$ is a suitable poset (see picture \ref{A(n,2)}). 

In \cref{sec:relenhancement}, we prove that (\ref{derivator intro}) is an enhancement of $\mathrm{D}(kA_n/I)$ and we explain the problem of enhancing the derived category of a quiver with relations via strict full subderivators.
\cref{sec5} is dedicated to our main result.

 \begin{theorem}[{Theorem \ref{mainthm}}]\label{intro:mainthm} There is an equivalence of stable derivators 
 \begin{equation}\label{intro}
\begin{tikzcd}
{\D^{A(n,2)}} \arrow[rr, "i^n", shift left=2] &  & \D^{A_n} \arrow[ll, "G^n", shift left=2]
\end{tikzcd}
      \end{equation}\label{eq:in}
where $i^n$ and $G^n$ are suitable compositions of left and right Kan extensions.
      \end{theorem}
The proof of this result involves equivalences of strict full subderivators which will often be taken care by \textbf{homotopical epimorphisms}: a technical tool introduced in \cite{gst:trees} which we discuss in \cref{sec:homoepi}.
Let us mention that this theorem is related to \cite[Corollary~9.15]{beckert19} where a different approach (involving hyperplanes) leads to an equivalent definition of $\D^{A(n,2)}$.

\medskip

While the first part of this article is dedicated to introduce and prove the main \cref{mainthm} which enhances a result in representation theory, \cref{sec:bridges} and \cref{sec:tilting} aim to explain how this result is closely related to homotopy theory. For this reason, the present work provides a bridge between these two areas. Indeed, we observe that equivalence (\ref{intro}) involves coherent chain complexes on the left hand side and filtered objects on the right hand side and this interpretation suggests a link with the $\infty$-Dold-Kan correspondence \cite[Theorem~1.2.4.1]{HA} (Theorem \ref{inftydold}). In this context, it is natural to investigate the relation between filtered objects and the $\infty$-category of coherent chain complexes \cite[Definition~35.1]{joyal:notes} arising from the generalization of the classical Dold-Kan correspondence \cite[Remark~1.2.4.3]{HA} (\cref{ari$4.7$}). This question was already answered by Ariotta in \cite[Theorem~4.7]{ariotta}. Since there is a canonical way to associate a stable derivator to a stable $\infty$-category (\cref{exder}), we get the following result.
\begin{proposition}[{\cref{relation}}] If we restrict to bounded chain complexes and bounded filtrations then
\cref{intro:mainthm} (\cref{mainthm}) is the derivator-theoretical version of \cite[Theorem~4.7]{ariotta} (\cref{ari$4.7$}).
\end{proposition}
\noindent In particular, we are able to see the relation between these results through the refinement of the mesh category described by Groth and {\v S}{\v t}ov{\'\i}{\v c}ek in \cite[Theorem~4.6]{[3]} (Theorem \ref{thm:AR}).  
Finally, in \cref{sec:tilting} we see how \eqref{eq:in} is an equivalence given by an \emph{universal tilting bimodule} \cite[Section~10]{[3]}. Namely, since every stable derivator is a closed module over the derivator of spectra \cite{cisinskitabuada11} and thanks to the tilting theory for derivators \cite[Theorem~8.5]{[3]} (Theorem \ref{specact}), the functor $G^n$ in (\ref{intro}) can be written as a \emph{canceling tensor product}, as follows.
\begin{proposition}[{\cref{tensor}}] The following equivalence of functors holds
    \[ G^n \cong T_n\otimes_{[A_n]}-.\]
\end{proposition}
Here $T_n$ is a spectral bimodule which we also explain how to compute. 

\subsection*{Acknowledgements}
I deeply thank my Ph.D. supervisor Jan {\v S}{\v t}ov{\'\i}{\v c}ek for his guidance, availability and for giving me the opportunity to work on this project. Moreover, I would like to thank Francesco Genovese for his help and support during the preparation of this article. I also thank Sebastian Opper for interesting discussions and Isaac Bird, Janina Letz and Jordan Williamson for useful comments on a preliminary version of the manuscript. Finally, I would like to thank the anonymous referee for his thorough review and for giving me a different point of view on my work.

\section{Preliminaries on derivators}
\label{sec:review}
Let us start by recalling some basics about the theory of derivators. More details can be found in  \cite{derptderstder, gps:mayer}. 
Let $\cCat$ be the 2-category of small categories and $\cCAT$ the 2-category of large categories. The concept of a '2-category' CAT presents the usual set theoretical issues, since it does not form a category enriched over Cat. However, since this subtlety is not relevant to our purposes here, we freely use expressions like '2-category CAT' as convenient shorthand. 

\begin{defn}
A 2-functor $\D:\cCat\op\to\cCAT$ is called \textbf{prederivator}. 
\end{defn}

\noindent A typical example is the \textbf{represented prederivator} which is associated to a bicomplete category $C \in \cCAT $
\begin{equation}\label{reprder}
\D_C \colon A \mapsto C^A
\end{equation}

\noindent where $C^A$ is the category of functors from $A$ to $C$. Prederivators form a 2-category $\cPDER$ whose morphisms are pseudo-natural transformations and transformations of prederivators are modifications. 

\noindent Given a prederivator $\D$ and a functor $u: A \to B$ in $\cCat$, the application of $\D$ to $u$ gives two categories $\D(A), \D(B)$ and a functor
\[\D(u)=u^*: \D(B) \to \D(A)\]

\noindent which is called \textbf{restriction} along $u$. Similarly, given two functors $u,v:A \to B$ and a natural transformation $\alpha: u \to v$, by applying $\D$, we get an induced natural transformation \[\alpha^*: u^* \to v^*.\]

\noindent Let now $e \in \cCat$ be the terminal category, \ie the category with one object and identity morphism only. For an object $a \in A$, we denote by $a: e \to A$ the unique functor sending the object of $e$ to $a$. The restriction along this functor \[a^*: \D(A) \to \D(e)\] is called \textbf{evaluation} and it takes values in the \textbf{underlying category} $\D(e)$. For two objects $X,Y$ and a morphism $f:X \to Y$ in $\D(A)$ we denote by \[f_a:X_a \to Y_a\] its image under $a^*$. Given a morphism $f\colon a \to b$ in $A$, this gives a natural transformation from $a : e \to A$ to $b: e \to A$ and so we have a natural transformation: $ f^*:a^* \to b^*$ in $\cCAT$. We call $\D(A)$ the category of \textbf{coherent $A$-shaped diagrams} in $\D$. Given an object $X\in\D(A)$, the natural transformations $f^*$ allows us to define a functor \[\mathrm{dia}_A(X)\colon A\to\D(e)\]which assigns, to any element of $A$, the object $X_a$. We call this map \textbf{underlying incoherent diagram}. Underlying incoherent diagrams yield a functor
\begin{align*}
    \mathrm{dia}_A\colon \D(A) & \to\D(e)^A\\
           X & \mapsto \mathrm{dia}_A(X). 
          \end{align*}
which in general is not an equivalence as coherent diagrams cannot be determined by their underlying diagrams, not even up to isomorphism. To solve this problem one can require (homotopy) completeness properties which we will express through \textbf{Kan extensions}.
          
\begin{defn}
 Let $u: A \to B \in \cCat$ and consider the restriction functor $u^\ast\colon\D(B)\to\D(A)$. When they exist, we call the left adjoint of the restriction \[u_!\colon \D(A)\to\D(B)\] \textbf{left Kan extension}
and the right adjoint 
\[u_\ast\colon\D(A)\to\D(B)\]
\textbf{right Kan extension}.
\end{defn}

When $B=e$, we have a unique functor $\pi=\pi_A\colon A\to e$ and we observe that the left Kan extension $\pi_!=\mathrm{colim}_A$ is a \textbf{colimit functor} and the right Kan extension $\pi_\ast=\mathrm{lim}_A$ is a \textbf{limit functor}. For the represented prederivator (\ref{reprder}) Kan extensions exist because they exist for bicomplete categories and, in particular, they can be calculated pointwise (see \cite[X.3.1]{maclane}). This is not the only example of a prederivator for which we can explicitly write a formula to compute the Kan extensions. Then, in order to give such formulas in a wider generality, it is useful to depict canonical transformations through \textit{squares}, as follows. Let $\D$ be a prederivator, assume the Kan extensions always exist and consider natural transformation $\alpha\colon up\to vq \in \cCat$. We can depict $\alpha$ with the following diagram
\begin{equation}
  \vcenter{\xymatrix{
      D\ar[r]^p\ar[d]_q \drtwocell\omit{\alpha} &
      A\ar[d]^u\\
      B\ar[r]_v &
      C
    }}.\label{eq:hoexactsq'}
\end{equation}
 Thanks to the adjunction unit $\eta$ and counit $\epsilon$, we get the so called \textbf{canonical mate transformations}
\begin{gather}
  q_! p^* \stackrel{\eta}{\to} q_! p^* u^* u_! \stackrel{\alpha^*}{\to} q_! q^* v^* v_! \stackrel{\epsilon}{\to} v^* u_!  \mathrlap{\qquad\text{and}}\label{eq:hoexmate1'}\\
  u^* v_* \stackrel{\eta}{\to} p_* p^* u^* v_* \stackrel{\alpha^*}{\to} p_* q^* v^* v_* \stackrel{\epsilon}{\to} p_* q^*.\label{eq:hoexmate2'}
\end{gather}

In particular, we are interested in canonical mate transformations arising from  \textbf{slice squares}. Here we first recall the definition of a \textbf{slice category}.
\begin{defn}
    Let $u:A \to B$ be a functor in $\cCat$ and $b$ an object in $B$. The \textbf{slice category} $(u/b)$ consists of pairs $(a,f)$ where $a$ is an object in $A$ and  $f\colon u(a)\to b$ a morphism in $B$.  A morphisms between two objects $(a,f)$ and $(a',f')$ is a morphism $a \to a'$ in $A$ making the following triangle commute in $B$:
 
\[\begin{tikzcd}
u(a) \arrow[d, "f"'] \arrow[r] & u(a') \arrow[ld, "f'"] \\
b                             &                    
\end{tikzcd}.\]
The slice category $(b/u)$ is defined dually.
\end{defn}
\textbf{Slice squares} are of the form
\begin{equation}
\vcenter{
\xymatrix{
(u/b)\ar[r]^-p\ar[d]_-{\pi_{(u/b)}}\drtwocell\omit{}&A\ar[d]^-u&&(b/u)\ar[r]^-q\ar[d]_-{\pi_{(b/u)}}&A\ar[d]^-u\\
e\ar[r]_-b&B,&&e\ar[r]_-b&B\ultwocell\omit{},
}
}
\label{eq:Der4}
\end{equation}
and they come with canonical transformations $u\circ p\to b\circ\pi$ and $b\circ\pi\to u\circ q$. Here the functor $p\colon (u/b)\to A$ denotes the projection onto the first component and $q$ is defined dually. We are now ready to give the general definition of a prederivator which satisfies certain (homotopy) completeness conditions i.e. a \textbf{derivator}.

\begin{defn}\label{defn:derivator}
  A prederivator $\D\colon\cCat\op\to\cCAT$ is called \textbf{derivator} if it satisfies the following axioms:
  \begin{itemize}
  \item[(Der1)] $\D\colon \cCat\op\to\cCAT$ takes coproducts to products, i.e., the canonical map \[\D(\coprod A_i)\to\prod\D(A_i)\] is an equivalence.  In particular, $\D(\emptyset)$ is equivalent to the terminal category.
  \item[(Der2)] For any $A\in\cCat$, a morphism $f\colon X\to Y$ in $\D(A)$ is an isomorphism if and only if the morphisms $f_a\colon X_a\to Y_a$, for any $a\in A$, are isomorphisms in $\D(e).$
  \item[(Der3)] For every functor $u \colon A \to B$, there exist both the left Kan extension $u_!$ and right Kan extension $u_*$ of the restriction $u^*$.
  \item[(Der4)] For any functor $u\colon A\to B$ and any object $b\in B$, the canonical mate transformations associated to the slice squares \eqref{eq:Der4} are isomorphisms.
  \end{itemize}
\end{defn}

\noindent By axioms (Der$1$) and (Der$3$), it follows that $\D(A)$ has small categorical products and coproducts for each small category $A$, so in particular it has initial and terminal objects.

\smallskip

\noindent Derivators form a $2$-category $\cDER$ which is a full sub-$2$-category of $\cPDER$: morphisms of derivators are simply morphisms of underlying prederivators and, similarly, natural transformations are modifications.

\begin{example}\label{exder} Let us list some relevant examples of derivators.
\begin{itemize}
    \item[(1)]The represented prederivator $\D_C$ \eqref{reprder} is itself a derivator. Indeed, the Kan extensions $u_!,u_\ast$ are then the ordinary Kan extension functors and the underlying category is isomorphic to $C$ itself.
    \item[(2)] Let $\mathbf{C}$ be a \emph{Quillen model category} (see e.g.~\cite{quillen:ha,hovey:model}) with \emph{weak equivalences} $\mathbf{W}$. We define the underlying \textbf{homotopy derivator} $\mathscr{H}\!\mathit{o}(\mathbf{C})$ by formally inverting the pointwise weak equivalences \[\mathscr{H}\!\mathit{o}(\mathbf{C}):A\mapsto  (\mathbf{C}^A)[(\mathbf{W}^A)^{-1}]\] 
    (see \cite{derptderstder} if $\mathbf{C}$ is combinatorial and \cite{cisinskidircoho} for a more general proof). The Kan extensions of $\mathscr{H}\!\mathit{o}(\mathbf{C})$  are the derived versions of the Kan extension of $\D_{\mathbf{C}}$ and the underlying category of $\mathscr{H}\!\mathit{o}(\mathbf{C})$ is the homotopy category $\Ho(\mathbf{C})=\mathbf{C}[\mathbf{W}^{-1}]$ of $\mathbf{C}$. 
\item[(3)] Let $\cc$ be a bicomplete $\infty$-category $\cc$ in the sense of Joyal \cite{joyal:barca} and Lurie \cite{HTT} (see \cite{groth:scinfinity} for an introduction). We define the prederivator $\Ho_{\mathcal{C}}$ by
\[\Ho_{\mathcal{C}}\colon A \mapsto \Ho(\cc^{N(A)})\] where $N(A)$ is the nerve of $A$. For bicomplete $\infty$-categories this yields the \textbf{homotopy derivator} of $\cc$; a sketch of the proof for this fact can be found in~\cite{gps:mayer}.
\end{itemize}
\end{example}

The axioms of derivators allow us to define new derivators out of given ones.
\begin{proposition}\label{shifted}
Let $\D,\E$ be derivators, $A,B \in \cCat$. The following functors
    \[\D^{B}:A \mapsto \D(B\times A) \quad \text{and} \quad \D\op:A \mapsto \D(A\op)\op.
    \] are both derivators, called respectively \textbf{shifted derivator} \cite[Theorem~1.25]{derptderstder} and  

 \textbf{opposite derivator} \cite[Example~1.11]{derptderstder}.

\end{proposition}

\noindent The shifted derivator is the construction we will mostly work with because it allows us to study the homotopy theory of coherent diagrams of shape $B$ in $\D$.

\smallskip

As in (Der$4$), the request that certain canonical maps are isomorphisms often appears while working with derivators. It is then useful to introduce the notion of \textbf{homotopy exact square}: a square as in \eqref{eq:hoexactsq'} is \textbf{homotopy exact} if, for every derivator $\D$, the canonical mates \eqref{eq:hoexmate1'} and \eqref{eq:hoexmate2'} are isomorphisms. In particular, it is possible to show that \eqref{eq:hoexmate1'} is an isomorphism if and only if this is the case for \eqref{eq:hoexmate2'}. Homotopy exact squares are compatible with pasting \ie the passage to canonical mates \eqref{eq:hoexmate1'} and \eqref{eq:hoexmate2'} is functorial with respect to horizontal and vertical pasting (see p.$327$ in \cite{derptderstder}). Consequently, horizontal and vertical pastings of homotopy exact squares are homotopy exact \cite[Lemma~1.14]{derptderstder}. Other examples of homotopy exact squares can be found in \cite{maltsiniotis:htpy-exact} and \cite{derptderstder,gps:mayer,gst:basic}.

\begin{proposition}\label{kanextful}
The following are fundamental properties of Kan extensions.
    \begin{enumerate}
    \item Kan extensions along fully faithful functors are fully faithful. Namely, let $u\colon A\to B$ be a fully faithful functor, then the unit $\eta\colon \id\to u^\ast u_!$ and the counit $\epsilon\colon u^\ast u_\ast\to id$ are isomorphisms  \cite[Proposition~1.20]{derptderstder}.
    \item Kan extensions and restrictions in unrelated variables commute. Namely, given two functors $u\colon A\to B$ and $v\colon C\to D$ then the commutative square
\begin{equation}
\vcenter{
\xymatrix{
A\times C\ar[r]^-{u\times \id}\ar[d]_-{ \id\times v}&B\times C\ar[d]^-{\id\times v}\\
A\times D\ar[r]_-{u\times \id}&B\times D
}}
\end{equation}
is homotopy exact, i.e., in every derivator the canonical mate transformations $(\id \times v)_!(u \times \id)^* \to (u \times \id)^*(\id \times v)_!$ and $(u \times \id)^*(\id \times v)_* \to (\id \times v)_*(u \times \id)^*$ are isomorphisms \cite[Proposition~2.5]{derptderstder}. 
\item Right adjoint functors are \textbf{homotopy final}. If $u\colon A\to B$ is a right adjoint, then the square
\[
\xymatrix{
A\ar[r]^-u\ar[d]_-{\pi_A}\drtwocell\omit{\id}&B\ar[d]^-{\pi_B}\\
e \ar[r]_-\id& e
}
\]
is homotopy exact i.e., the canonical mate $\mathrm{colim}_Au^\ast\to\mathrm{colim}_B$ is an isomorphism \cite[Proposition~1.18]{derptderstder}. In particular, if $b\in B$ is a terminal object, then there is a canonical isomorphism $b^\ast\cong\colim_B$.
    \end{enumerate}
    \end{proposition}

\noindent Since Kan extensions and restrictions in unrelated variables commute, we have parametrized versions of restriction and Kan extension functors. Namely, for a derivator $\D$ and a functor $u\colon A\to B$, there are adjunctions of derivators
\begin{equation}
(u_!,u^\ast)\colon\D^A\rightleftarrows\D^B\qquad\text{and}\qquad
(u^\ast,u_\ast)\colon\D^B\rightleftarrows\D^A.\label{eq:Kan-adjunction}
\end{equation}
which are defined internally to the 2-category $\cDER$ \cite{[3],derptderstder}. In particular, if we restrict these adjunctions of derivators to the underlying categories, we obtain the corresponding adjunctions given by (Der$3$):
\begin{equation}
(u_!,u^\ast)\colon\D(A)\rightleftarrows\D(B)\qquad\text{and}\qquad
(u^\ast,u_\ast)\colon\D(B)\rightleftarrows\D(A).\label{eq2:Kan-adjunction}
\end{equation}
Moreover, if $u$ is fully faithful then in both the above cases, $u_!,u_\ast$ are fully faithful as well. This observation is relevant because it implies that they induce equivalences of derivators onto their essential images. We also have the following proposition.

\begin{proposition}[{\cite[Corollary~2.6]{derptderstder}}]\label{prodfunct} Let $\D$ be a derivator, $C$ a small category and let $u:A \to B$ be a fully faithful functor. An object $X \in \D^C(B)$ lies in the essential image of $u_!:\D^C(A) \to \D^C(B)$ if and only if $X_c$ lies in the essential image of $u_!:\D(A) \to \D(B)$ for all $c \in C$.
\end{proposition}

Finally, given a (pre)derivator $\D$, we write $X\in\D$ to indicate that there is a small category~$A$ such that $X\in\D(A)$.

\subsection{Stable derivators} \label{sec:stable}
In this article we aim to enhance an equivalence of triangulated categories. Higher categorical enhancements of triangulated categories have the additional properties of being \emph{stable}, which means that they in some sense behave like abelian categories. In this subsection we then recall some basics about stable derivators.

\begin{defn} 
A derivator $\D$ is \textbf{pointed} if $\D(e)$ has a zero object. 
\end{defn}

\begin{proposition}[{\cite[Proposition~3.2]{derptderstder}}]
    If $\D$ is pointed then so are the shifted derivators $\D^B$ and its opposite $\D\op$. 
\end{proposition}
\noindent In particular, for any $A \in \cCat$, $\D(A)$ have zero objects which are preserved by restriction and Kan extension functors. 
If $\D$ is a pointed derivator, some inclusion functors of small categories are especially interesting because we can describe the image of their Kan extensions.

\begin{defn}
    A functor $u\colon A\to B$ is a \textbf{sieve} if it is fully faithful and if for any morphism $b\to u(a)$ in $B$ there exists an $a'\in A$ with $u(a')=b$. A \textbf{cosieve} is defined dually.
\end{defn}

\noindent Note that, by \cref{kanextful}, if $u: A \to B$ is a (co)sieve then the Kan extensions $u_!, u_*$ are fully faithful.

Sieves and cosieves are what we need to realize \textbf{extensions by zero objects}.

\begin{proposition}[{\cite[Proposition~1.23]{derptderstder}}]\label{extbyzero}
Let $\D$ be a pointed derivator. 
\begin{enumerate}
\item Let $u:A \to B$ be a cosieve. Then the left Kan extension $u_!$ is fully faithful and $X \in \D(B)$ lies in the essential image of $u_!$ if and only if $X_b$ is zero for all $b \in B \setminus u(A)$.
\item Let $u:A \to B$ be a sieve. Then the Kan extension $u_*$ is fully faithful and $X \in \D(B)$ lies in the essential image of $u_*$ if and only if $X_b$ is zero for all $b \in B \setminus u(A)$.
\end{enumerate}
\end{proposition}
\noindent
 In particular, when $u$ is cosieve we call the functor $u_!$ \textbf{left extension by zero} and, when $u$ is a sieve, we call $u_\ast$ \textbf{right extension by zero}.

\medskip
 
Stable derivators are pointed derivators with an additional property. Let $[n]$ be the poset $(0<\cdots<n)$ considered as a category. The commutative square $\square=[1]\times[1],$
\begin{equation}
\begin{tikzcd}
{(0,1)} \arrow[r]           & {(1,1)}           \\
{(0,0)} \arrow[u] \arrow[r] & {(1,0)} \arrow[u] 
\end{tikzcd}
\label{eq:square}
\end{equation}

\noindent comes with full subcategories $i_\llcorner\colon\llcorner\to\square$ and $i_\urcorner\colon\urcorner\to\square$ obtained by removing the terminal object and the initial object, respectively. Since both inclusions are fully faithful, so are $(i_\llcorner)_!\colon\D^\llcorner\to\D^\square$ and $(i_\urcorner)_\ast\colon\D^\urcorner\to\D^\square$. 

\begin{defn}
    A square $X\in\D^\square$ is \textbf{cocartesian} if it lies in the essential image of $(i_\llcorner)_!$ and it is \textbf{cartesian} if it lies in the essential image of $(i_\urcorner)_\ast$. A square which is both cartesian and cocartesian is called \textbf{bicartesian}.
\end{defn}

\noindent In the proof of \cref{mainthm}, we often deal with checking whether a square contained in a larger diagram is (co)cartesian. This is possible thanks to the following technical Proposition.

\begin{proposition}[{\cite[Proposition~3.10]{derptderstder}}]\label{ccsq} 
Let $i:\square \to B$ be a square in $B$ and let $u:A \to B$ be a functor.\begin{enumerate}
\item Assume that the induced functor $ \llcorner \stackrel{\tilde{i}}{\rightarrow}(B\setminus i(1,1))_{/i(1,1)}$ has a left adjoint and that $i(1,1)$ does not lie in the image of $u$. Then for all $X=u_!(Y) \in \D(B)$, $Y\in\D(A)$, the induced square $i^*(X)$ is cocartesian. \item Assume that the induced functor $ \urcorner \stackrel{\tilde{i}}{\rightarrow}(B\setminus i(0,0))_{i(0,0)/}$ has a right adjoint and that $i(0,0)$ does not lie in the image of $u$. Then for all $X=u_*(Y) \in \D(B)$, $Y\in\D(A)$, the induced square $i^*(X)$ is cartesian. 
\end{enumerate}

\end{proposition}

\begin{remark}
    By \cref{kanextful},  we have that if $u\colon A\to B$ is a fully faithful functor, then the counit $\epsilon\colon u^\ast u_\ast\to \id$ in an isomorphism. Then an object $X \in \D(B)$ belongs to the essential image of $u_*$ if and only if $\eta(X): X \to u_*u^*(X)$ is an isomorphism. Dually, the same property holds for $u_!$. As a consequence, by \cref{ccsq}, when $B \setminus \essIm(u)$ contains only $i(0,0)$, we have that the essential image of $u_*$ (respectively of $u_!$) consists of all the objects $X$ such that $i^*(X)$ is cartesian (respectively cocartesian).
\end{remark}

\oneline

\begin{defn}
A pointed derivator is \textbf{stable} if the classes of cartesian squares and cocartesian squares coincide. Recall that such squares are then called \textbf{bicartesian}.
\end{defn}
\noindent Different characterizations of stable derivators are given in \cite[Theorem~7.1]{gps:mayer} and \cite[Corollary~8.13]{gst:basic}.

 \begin{proposition}[{\cite[Proposition~4.3]{derptderstder}}]
 If $\D$ is a stable derivator then so are the shifted and opposite derivators $\D^B$ and $\D\op$.
 \end{proposition}
  
\begin{example}\label{egs:stable} The following are examples of stable derivators which are useful for this work.
\begin{enumerate}
\item Let $\mathcal{G}$ be a Grothendieck category. We have a stable combinatorial model category for complexes over a Grothendieck category and quasi-isomorphisms as weak equivalences \cite[Example~3.11]{sto:dercot}. Recall that the derived category $\mathrm{D}(\mathcal{G})$ is the localization of the category of chain complexes at the class of quasi-isomorphisms. Then, by \cref{exder} and \cite[Proposition~1.30]{derptderstder} the \textbf{derivator associated to a Grothendieck  category} is the $2$-functor \[\D_{\mathcal{G}}\colon A \mapsto \mathrm{D}(\mathcal{G}^A)\]  and, since the model category is stable, the derivator is such.
In particular, if we choose $\mathcal{G}$ to be the module category of an algebra over a quiver, we then use this example to enhance \cref{dercat}.
\item There are many Quillen equivalent stable model categories of spectra such that the homotopy category is the stable homotopy category $\mathcal{SHC}$. The homotopy derivator $\cSp$ associated to any of these model categories is stable. We will refer to it as the \textbf{derivator of spectra}, it will play an essential role in \cref{sec:tilting}.
\item  Homotopy derivators of stable $\infty$-categories and stable model categories are stable \cite{gps:mayer}.

More examples can be found in \cite{gst:basic}.
\end{enumerate}
\end{example}

The following result justifies our interest in the notion of stable derivators. 

\begin{theorem}
    [{\cite[Theorem~4.16]{derptderstder}}]
    If $\D$ is a strong and stable derivator then its underlying category $\D(e)$ is triangulated. 
\end{theorem}

\noindent The additional hypothesis of being \textbf{strong} requires that, for any $A \in \cCat$, the partial underlying diagram functor \[\D(A \times [1]) \to D(A)^{[1]}\]

\noindent is full and essentially surjective. The example of derivators considered in the present work, such as homotopy derivators of model categories and $\infty$-categories, are all strong. This concept is then not essential for our purposes.

\noindent Let us describe the triangulated structure of the underlying category, provided the derivator is strong and stable. Given a bicartesian square $X \in \D^{\square}$, if $X_{(0,1)}$ is the zero object then  \[X_{(0,0)} \to X_{(1,0)} \to X_{(1,1)} \to \] is a triangle in $\D(e)$. In particular, $X_{(1,1)}$ is the cone of the morphism $X_{(0,0)} \to X_{(1,0)}$.

\medskip

\noindent We conclude this subsection with the following definition which we need to state \cref{mainthm}.

\begin{defn}
    A morphism of derivators is \textbf{right exact} if it preserves initial objects and cocartesian squares. Dually we define \textbf{left exact} morphism. A morphism which is both right and left exact is called  \textbf{exact}. 
\end{defn}
\noindent For stable derivators these three notions clearly coincide. 

\subsection{Total cofiber construction}

Let us recall the \textbf{total cofiber construction} because this is the key idea behind the proof of \cref{mainthm}. More details can be found in \cite{deriteratedcones}.

\begin{construction}\label{tcof}
Consider $\llcorner=\square-\{(1,1)\}$ the full subcategory of the square obtained by removing the final object and consider the category $\tilde{K}^3_{1,2}$ in the diagram below. This is the cocone on the square obtained by adjoining a new terminal object $(2,1)$.
\[
\begin{tikzcd}
                            &                    & {(2,1)} \\
{(0,1)} \arrow[r]           & {(1,1)} \arrow[ru] &         \\
{(0,0)} \arrow[u] \arrow[r] & {(1,0)} \arrow[u]  &        
\end{tikzcd}.
\]

\noindent Associated to this category are the fully faithful inclusions of the source and target square 
\begin{equation}\label{eq:t-s-squares}
s=s_\llcorner\colon\square\to \tilde{K}^3_{1,2}\qquad\text{and}\qquad t=t_\llcorner\colon\square\to \tilde{K}^3_{1,2},
\end{equation}

\noindent where the image of $s$ is given by all objects except $(2,1)$ and the image of $t$ is given by all objects except $(1,1)$.
\end{construction}

\begin{proposition}[{\cite[Proposition~2.2]{deriteratedcones}}]\label{prop:cocart-cocone}
Let $\D$ be a derivator and let $s,t\colon\square\to \tilde{K}^3_{1,2}$ be the inclusions of the source and target squares.
\begin{enumerate}
\item The morphism $t_!\colon\D^\square\to\D^{\tilde{K}^3_{1,2}}$ is fully faithful and $Y\in\D^{\tilde{K}^3_{1,2}}$ lies in the essential image of~$t_!$ if and only if the source square $s^\ast Y$ is cocartesian.
\item A square $X\in\D^\square$ is cocartesian if and only if the canonical comparison map
\begin{equation}\label{eq:comp-map}
\mathrm{can}=\mathrm{can}(X)\colon t_!(X)_{(1,1)}\to t_!(X)_{(2,1)}.
\end{equation}
is an isomorphism.
\end{enumerate}
\end{proposition}

\begin{defn}[{\cite[Definition~2.4]{deriteratedcones}}]\label{defn:tcof}
Let $\D$ be a pointed derivator. The \textbf{total cofiber} of $X\in\D^\square$ is the cone of the comparison map~\eqref{eq:comp-map}. In formulas, we set
\[
\mathsf{tcof}(X)=C(\mathrm{can}(X))\in\D,
\]
where $C: \D^{[1]} \to \D$ is the cone morphism (see \cite[Subsection 3.3]{derptderstder}). The definition of the \textbf{total fiber} $\mathsf{tfib}(X)\in\D$ is dual.
\end{defn}

\section{Strict full subderivators}\label{strictfullsubder}

In this section we introduce the new notion of strict full subderivator. Strict full subderivators will allow us to enhance the derived category $\mathrm{D}(kA_n/I)$ (cf. \ref{derivedeq}), so in particular to enhance the relations given by the ideal $I$.

\noindent Let us start by recalling the already existing definition of full subprederivator. 

\begin{defn}
Let $\D$ be a derivator, a \textbf{full  subprederivator} $\D'$ of $\D$ is a full sub-$2$-functor of $\D$ \ie  it is a $2$-functor $\D' \colon\cCat\op\to\cCAT$ such that $\D'(I)\subseteq \D(I)$ is a full subcategory for any $I \in \cCat$. In particular, if $u\colon A\to B$ is a functor, then the restriction $(u')^*= \D '(u)\colon\D'(B) \to \D'(A)$ is given by $(u^*)_{\vert \D'(B)}$. In other words, the diagram
\[
\begin{tikzcd}
\D'(A) \arrow[rr, "i_A", hook]                       &  & \D(A)                  \\
\D'(B) \arrow[u, "(u')^*"] \arrow[rr, "i_B", hook] &  & \D(B) \arrow[u, "u^*"]
\end{tikzcd}\]
 needs to be commutative. Here $i_A$ and $i_B$ are the inclusion functors. 
\end{defn}

\noindent The definition of \textbf{full subderivator} naturally follows: namely, it is a full subprederivator which is also a derivator. Given $\D'$ a full subderivator of $\D$, by (Der3) for every restriction functor $(u')^*$, there exist both the left $(u')_!$ and the right $(u')_*$ Kan extensions. What cannot be guaranteed is that these Kan extensions are compatible with the Kan extensions of $\D$. In order to ensure this compatibility condition, we introduce the notion of a \textit{strict full subderivator} \ie  
a full subderivator which preserves Kan extensions.

\begin{defn} 
A \textbf{strict full subderivator} $\D'$ of a derivator $\D$ is a full subprederivator satisfying (Der$1$), (Der$3$) and such that the left and right Kan extensions are given by the left and right Kan extensions of $\D$ restricted to $\D'$. Namely, if $u\colon A\to B$ is a functor, then $(u')_!,(u')_*\colon\D'(A) \to \D'(B)$ are given by $u_{!_{\vert \D'(A)}},\, u_{*_{\vert \D'(A)}}$ respectively.
\end{defn}

For a strict full subderivator, we then have the following commutative diagram 
\[
\begin{tikzcd}
\D'(A) \arrow[d, "(u')_*", bend left, shift left=6] \arrow[rr, "i_A", hook] \arrow[d, "(u')_!"', bend right, shift right=7] &  & \D(A) \arrow[d, "u_*", bend left, shift left=6] \arrow[d, "u_!"', bend right, shift right=6] \\
\D'(B) \arrow[u, "(u')^*"] \arrow[rr, "i_B", hook]                                                                          &  & \D(B) \arrow[u, "u^*"]                                                                      
\end{tikzcd},
\]
where $i$ is the inclusion functor, the restrictions strictly commute and the Kan extensions commute as well.

\begin{proposition}
Any strict full subderivator $\D' \subseteq \D$ is a derivator. 
\end{proposition}
\begin{proof}
Let us show that the two additional axioms are satisfied.

\begin{itemize}

\item[(Der$2$)] Consider a small category $A \in \cCat$, and let $f:X \to Y$ be a morphism in $\D'(A)$. If $f$ is an isomorphism then its image under the evaluation functor $(a')^*:\D'(A)\to\D'(e)$ is also an isomorphism, as any functor preserves isomorphisms. Hence $(a')^*(f)=f_a :X_a \to Y_a$ is an isomorphism. On the other hand, for any $a \in A$, by definition we have the following commutative diagram
\[\begin{tikzcd}
\D'(A) \arrow[rr, "i_A", hook] \arrow[d, "(a')^*"'] &  & \D(A) \arrow[d, "a^*"'] \\
\D'(e) \arrow[rr, "i_e", hook]                      &  & \D(e)                  
\end{tikzcd}.
\]
If $(a^*)'(f)$ is an isomorphism then also $i_e(a^*)'(f)=a^*(i_A(f))$ is an isomorphism. Being a derivator, $\D$ satisfies (Der$2$), so $i_A(f)$ and hence $f$ are isomorphisms.
\item[(Der$4$)] By the property (Der$4$) of the derivator $\D$, we have an isomorphism $\pi_!p^* \xrightarrow{\sim} b^*u_!$. By composing with the inclusion functor, we get the map $\pi_!p^*i_A \xrightarrow{\sim} b^*u_!i_A$ which is again an isomorphism. We have the following diagram
\[ \begin{tikzcd}
\pi_!p^*i_A \arrow[rr, "\sim"] \arrow[d]        &  & b^*u_!i_A \arrow[d]    \\
\pi_!i_{u/b}(p^*)' \arrow[d] \arrow[rr, "\sim"] &  & b^*i_B(u_!)' \arrow[d] \\
i_e(\pi_!)'(p^*)' \arrow[rr, "\sim"]            &  & i_e(b^*)'(u_!)'       
\end{tikzcd} \]
where all the vertical arrows are equivalences. The last isomorphism implies the desired one, which is $(\pi_!)'(p^*)' \rightarrow (b^*)'(u_!)'$.
\end{itemize}
\end{proof}

In this paper we focus on working with strict full subderivators of shifted derivators. The following construction provide us with a large class of examples of those.

\begin{construction}\label{constr:strictfullsubder}
    Let $\D$ be a derivator, $A \in \cCat$ and consider the following data:
\begin{itemize}
    \item[(1)] A functor $j:C \to A \in \cCat$.
    \item[(2)] Two fully faithful inclusions  $k_l: B_l \to C$ and $k_r: B_r \to C$ such that the essential images of the respectively left and right Kan extensions coincide, 
    \begin{equation}\label{essIm}
        \essIm(k_l)_!=\essIm(k_r)_* \end{equation}
    in $\D(C)$.
\end{itemize}
By \cref{kanextful}, the equality \eqref{essIm} implies that
\begin{equation}
    \essIm(k_l \times \id_I)_!=\essIm(k_r \times \id_I)_*  \in \D(C \times I),
\end{equation}
for any $I \in \cCat.$ Then let us define
\begin{equation}\label{eq:strictfull}
    \E(I)\coloneqq\left\{X \in \D^A(I)=\D(A \times I): (j \times \id)^*(X) \in \essIm(k_l \times \id)_!  
     \right\} \subseteq \D^A(I).
\end{equation}
\end{construction}

We now prove that $\E(I)$ are actually values of a strict full subderivator $\E \subseteq \D^A$. 
\begin{proposition}\label{prop:fullsubder}
    Let $\D$ be a derivator, $A \in \cCat$ and $j, k_l$ and $k_r$ as in \cref{constr:strictfullsubder}. Then there is a strict full subderivator $\E \subseteq \D^A$ such that for any $I \in \cCat$, the category $\E(I)$ is given by \eqref{eq:strictfull}.
    \end{proposition}
\begin{proof}
 For simplicity, let us denote $(\id_A \times u)$ by $u_A$ for any $A \in \cCat$ and $(j \times \id_I),\, (k_l \times \id_I)$ respectively by $j_I,\, k^I$. First, let us prove that $\E$ is a subprederivator \ie  given a functor $u:J \to I$, let us verify that if $X \in \E(I)$ then $u^*_A(X) \in \E(J)$. We have the following commutative diagram
\[\begin{tikzcd}
\D^I(B_l) \arrow[rr, "u^*_{B_l}"] \arrow[d, "k^I_!"'] &  & \D^J(B_l) \arrow[d, "k^J_!"]  \\
{\D^I(C)} \arrow[rr, "{u^*_{C}}"]           &  & {\D^J(C)}                 \\
\D^I(A) \arrow[u, "j^*_I"] \arrow[rr, "u^*_A"]  &  & \D^J(A) \arrow[u, "j^*_J"']
\end{tikzcd},\]
where we identify $\D^B(A)$ with $\D^A(B)$ via the natural isomorphism $B \times A \cong A \times B$ for any $A,B\in\cCat$. By definition, it suffices to prove that if $X\in \D^I(A)$ is such that there exists $Y \in \D^I(B_l)$ and $j^*_I(X) \cong k^I_!(Y)$, then $u^*_A(X) \in \D^J(A)$ is such that $j^*_J(u^*_A(X)) \in \essIm k^J_!$. This is guaranteed by the following isomorphisms
\[ j^*_J(u^*_A(X))=u^*_{C}j^*_I(X)\cong u^*_{C}k^I_!(Y) \cong k^J_!u^*_{B_l}(Y),
\]
where the last isomorphism is given by \cref{kanextful}.
To prove (Der$1$), it is enough to consider the commutative diagram 
\[\begin{tikzcd}
\D^{\coprod_iI_i}(B_l) \arrow[rr, "F^{B_l}"] \arrow[d, "k_!^{\coprod_i l_i}"'] &  & \prod_i\D^{I_i}(B_l) \arrow[d, "\prod_i k_!^{I_i}"]  \\
{\D^{\coprod_iI_i}(C)} \arrow[rr, "{F^{C}}"]           &  & {\prod_i\D^{I_i}(C)}                \\
\D^{\coprod_iI_i}(A) \arrow[u, "j^*_{\coprod_i I_i}"] \arrow[rr, "F^A"]  &  & \prod_i\D^{I_i}(A) \arrow[u, "\prod_i j^*_{I_i}"']
\end{tikzcd},
\]
where $F^e, F^{[1]}$, and $F^A$ are equivalences coming from the property (Der$1$) for the derivator $\D$. Then, we conclude the proof by diagram chasing, as in the previous case.
(Der$3$) follows by proving that the Kan extensions are the ones in $\D$ restricted to $\E$. It can be verified again by diagram chasing, thank to \cref{kanextful}
\[\begin{tikzcd}
\D^I(B_l) \arrow[rr, "u^{B_l}_!"] \arrow[d, "k^I_!"'] &  & \D^J(B_l) \arrow[d, "k^J_!"]  \\
{\D^I(C)} \arrow[rr, "{u^{C}_!}"]           &  & {\D^J(C)}                 \\
\D^I(A) \arrow[u, "j^*_I"] \arrow[rr, "u^A_!"]  &  & \D^J(A) \arrow[u, "j^*_J"'].
\end{tikzcd}\]
Analogously, we can prove the same for the right Kan extension by redefining $k^I$ as $(k_r \times \id_I)$ and considering $k^I_*$ and $u_*$.
\end{proof}

\begin{proposition}
If $\D$ is a stable derivator then $\E$ is a stable strict full subderivator.
\end{proposition}
\begin{proof}
    Since the Kan extensions of $\E$ are the restrictions of the ones in $\D^A$, the conclusion follows directly. 
\end{proof}

\begin{notation} Let us fix some notation for the classes of examples of strict full subderivators, which are of particular relevance in our context. 
 \begin{itemize} \item[(1)] The strict full subderivator spanned by all the coherent diagrams where we require an arrow to be an isomorphism. Namely, consider
     \begin{itemize}
         \item the inclusion functor $j: [1] \to A$ that chooses a morphism in $A$,
         \item the functor $k_l: e \to [1]$ that chooses the initial object and  $k_r: e \to [1]$ that chooses the final object.
     \end{itemize} 
     For any $I\in \cCat$, we define 
     \[\E^{is}_A(I)\coloneqq\left\{X \in \D^A(I)=\D(A \times I): (j \times \id)^*(X) \in \essIm(k_l \times \id)_!  
     \right\}.\]
     
      \item [(2)]\label{vanish} Let $\D$ be a pointed derivator. The strict full subderivator spanned by all the coherent diagrams which vanish in one position. Namely, consider
      \begin{itemize}
          \item the functor which chooses an object in $A$, $j:e\to A$.
          \item the canonical functor $k_l=k_r: \varnothing \empty \to e$. 
      \end{itemize}
       For any $I\in \cCat$, we define
     \[\E^{va}_A (I)\coloneqq\{X \in \D^A(I)=\D(A\times I): (j\times \id)^*(X)\in \essIm(k_l \times \id)_! \}.\]

     \item[(3)] Let $\D$ be a stable derivator. The strict full subderivator spanned by all the coherent diagrams where we require a square to be bicartesian. Namely, consider 
     \begin{itemize}
         \item the functor which chooses a commutative square in A, $j: \square \to A$,
         \item the inclusion functors $k_l: \llcorner \to \square$, $k_r: \urcorner \to \square$.
     \end{itemize}
     For any $I\in \cCat$, we define
     \[\E^{bi}_A(I)\coloneqq\{X \in \D^A(I)=\D(A \times I): (j \times \id)^*(X)\in \essIm(k_l \times \id)_!\}. \]
 \end{itemize}
\end{notation}

\noindent In order to write the derivator whose underlying category is $D(kA_n/I)$ (see \eqref{derivedeq}) and to prove \cref{mainthm}, we need to define strict full subderivators with more than one vanishing position, isomorphism arrow or bicartesian square. The following proposition suggests that it is actually enough to define them as intersections.

\begin{proposition}\label{prop:intersection} Intersection of strict full subderivators is a strict full subderivator. In particular, this is true for  $\E^{va}_A$, $\E^{is}_A$ and $\E^{bi}_A$.
\end{proposition}
\begin{proof} Let $\D$ be a derivator and $\{\D_k\}_{ k\in K}$ a family of strict full subderivators of $\D$. Since all the Kan extensions and restriction functors are just those restricted from $\D$, we only need to prove $(Der 1)$. It holds because of the following equivalences: 
\[\bigcap_k \D_k ( \coprod_i I_i)\cong\bigcap_k \prod_i \D_k(I_i) \cong \prod_i\bigcap_k\D_k(I_i)\]
where $\D_k$ strict full subderivator of $\D$, for any $k$. The first equivalence holds because $(Der 1)$ for $\D$ restricts to $\D_k$ for any $k$. The second equivalence holds because the product commutes with intersection of full subcategories.
\end{proof}

 \section{Homotopical epimorphisms}\label{sec:homoepi}
Let $\D$ be a derivator and, as in the previous section, consider the functor \[k: e \to [1] \in \cCat\] which chooses the initial object in $[1]$. The properties of the left Kan extension make sure that the underlying incoherent diagram of an object $Y$ in the essential image of $k_!: \D(e) \to \D([1])$ is an isomorphism.

\noindent Consider a small category $A$, a coherent diagram $X \in \D(A)$, and the functor \[j: [1] \to A\]
which chooses a morphism in $A$. The map we chose through the functor $j: [1] \to A$ is an \textbf{isomorphism in the coherent diagram} if $j^*(X)$ is in the essential image of the left Kan extension $k_!: \D(e) \to \D([1])$.

\noindent Let now $u: A \to B$ be the localization functor in $\cCat$ inverting the morphism at which $j$ points. Then it is not difficult to observe that the essential image of $u^*: \D^B \to \D^A$ is contained in the strict full subderivator $\E_A^{is}$. The natural question is when $u^*:\D^B \to \E_A^{is}$ is an equivalence. The answer is: when $u^*$ is an homotopical epimorphism.

\begin{defn}[{\cite[Definition~3.8]{homotopicaliso}}]\label{defn:htpy-epi}
A functor $u\colon A\to B$ is an \textbf{homotopical epimorphism} if for every derivator $\D$, the restriction functor $u^\ast\colon\D(B)\to\D(A)$ is fully faithful.
\end{defn}

\begin{remark}\label{homoepisq}
By \cref{kanextful}, $u\colon A\to B$ is an homotopical epimorphism if and only if the square
\[
\xymatrix{
A\ar[r]^-u\ar[d]_-u&B\ar[d]^-\id\\
B\ar[r]_-\id&B
}
\]
is homotopy exact.
\end{remark}

\begin{example}
    Let $A=[2]$ and $B=[1]$. Let $\E_A^{is}$ be the strict full subderivator of $\D^A$ where we require the arrow $1 \to 2$ to be an isomorphism. If we consider the functor
\begin{align*}
u \colon A & \to B \\
0 & \mapsto 0\\
1,2 & \mapsto 1,
\end{align*}
then, $u^*: \D^B \to \E_A^{is}$ is fully faithful and then it is an homotopical epimorphism. 
    
\end{example}

\noindent To understand whether a given functor is an homotopical epimorphism, there are some criteria which we now illustrate. Homotopical epimorphisms and the following criteria are among the fundamental tools which will allow us to prove \cref{mainthm}.

\begin{proposition}[{\cite[Proposition~8.2]{homotopicaliso}}]\label{prop:detect}
Let $u\colon A\to B$ be essentially surjective, let $\D$ be a derivator, and let $u^\ast\colon \D^B\to \D^A$ be the restriction morphism. Let us assume further that $\E\subseteq\D^A$ is a full subprederivator such that
\begin{enumerate}
\item the essential image $\essIm(u^\ast)$ lies in $\E$, i.e., $\essIm(u^\ast)\subseteq \E\subseteq \D^A$, and
\item the unit $\eta\colon X\to u^\ast u_! X$ is an isomorphism for all $X\in\E$.
\end{enumerate}
Then $u^\ast\colon\D^B\to\D^A$ is fully faithful and $\essIm (u^\ast)=\E$. In particular, $\E$ is a derivator.
\end{proposition}

\begin{construction}({\cite[Construction.~8.4]{homotopicaliso}})\label{con:colim-mor}
Let $\D$ be a derivator, $A\in\cCat$ and let $a\in A$. Associated to the square
\[
\xymatrix{
e\ar[r]^-a\ar[d]&A\ar[d]^-{\pi_A}\\
e\ar[r]&e
}
\]
there is the canonical mate 
\begin{equation}\label{eq:can-colim-map}
a^\ast\to\colim_A.
\end{equation}
As a special case relevant in later applications, given a functor $u\colon A\to B$ and $a\in A$ we consider the functor $p\colon (u/ua)\to A$. Whiskering the mate \eqref{eq:can-colim-map} in the case of $(a,\id\colon ua\to ua)\in(u/ua)$ with $p^\ast$ we obtain a canonical map
\begin{equation}\label{eq:lkan-colim-map}
a^\ast=(a,\id_{ua})^\ast p^\ast\to \colim_{(u/ua)}p^\ast.
\end{equation} 
\end{construction}

\begin{lemma}[{\cite[Lemma.~8.7]{homotopicaliso}}]\label{lem:detect-Kan-to-colim}
Let $\D$ be a derivator, $u\colon A\to B$, and $a\in A$. The component of the unit $a^\ast\eta\colon a^\ast\to a^\ast u^\ast u_!$ is isomorphic to \eqref{eq:lkan-colim-map}. In particular, $\eta_a$ is an isomorphism if and only if this is the case for \eqref{eq:lkan-colim-map}.
\end{lemma}

\noindent We will later apply the previous lemma to functors $v\colon C \to u/u(a)$ where $C \in \cCat$  and $ u/u(a)$ is a slice category which admits a terminal object. For this purpose we collect the following result.

\begin{lemma}[{\cite[Lemma.~8.8]{homotopicaliso}}] \label{lem:detect-final}
Let $u\colon A\to B$ be a functor in $\cCat$ and let $a\in A$.

If $A$ admits a terminal object $\infty$, then the map $a^\ast\to\colim_A$ \eqref{eq:can-colim-map} is naturally isomorphic to $a^\ast\to\infty^\ast$.
\end{lemma}

\begin{warning} Not every localization at a morphism is a homotopical epimorphism. 
    Consider the following poset, which we denote by $B$

\[
\begin{tikzcd}
a \arrow[d] \arrow[rd] & b \arrow[ld] \arrow[d] \\
c                      & d                     
.\end{tikzcd}\]

Let $\E_B^{is}$ be the strict full subderivator of $\D^{B}$ where we require the arrow $a \to c$ to be an isomorphism. We want to check whether the following functor is an homotopical epimorphism

\begin{equation*}
\begin{tikzcd}
a \arrow[d, Rightarrow] \arrow[rd] & b \arrow[d] \arrow[ld] & {} \arrow[r, "v", shift right=9] & {} & a \arrow[d] \arrow[rd]               & b \arrow[ld] \arrow[d] \\
c                                  & d                      &                                    &    & c \arrow[u, bend left, shift left=2] & d.                     
\end{tikzcd}
\end{equation*}

\noindent Here the double line arrow denotes the isomorphism in the coherent diagram and the target category is the localization of $B$ at the morphism $ a \to c$, we denote it by $B_{a\sim c}$. 

\noindent This functor is surjective on objects and $\essIm(v^\ast) \subseteq \E_B^{is}$. By \cref{prop:detect}, we then only need to check that the unit   $\eta\colon X\to v^\ast v_! X$ is an isomorphism for all $X\in \E_B^{is}$. 

\noindent Thanks to (Der$2$), it suffices to check that the map 
 \[i^*\eta\colon i^*X \to i^*v^*v_!X\]
 is an isomorphism for every object $i\in B$. By Lemma \ref{lem:detect-Kan-to-colim} we equivalently show that the map
 \begin{equation}\label{iso}
     i^*X\to \colim_{(v/v(i))}p^*X
  \end{equation}
 is an isomorphism for every $i\in B$. Due to the shape of $B_{a\sim c}$, the only interesting case is when $i=d$. Indeed, the slice category $v/v(d)$ looks as $B$ and there is a counterexample for the derivator of any field $k$ 
\begin{equation}\label{eq:counterexample}
\begin{tikzcd}
0 \arrow[d] \arrow[rd] & k \arrow[ld] \arrow[d] \\
0                      & 0                     
,\end{tikzcd}
\end{equation}
where the colimit of such a category is not, in general, isomorphic to the evaluation in $d$. Indeed, the colimit of \eqref{eq:counterexample} is $\Sigma k$, the suspension of $k$.
Consequently, $v$ is not an homotopical epimorphism and it shows then a case where the localization at a morphism is not an homotopical epimorphism.  
\end{warning}
\section{Enhancement of a quiver with relations}\label{sec:relenhancement}

 The main goal of the present work is to enhance \cref{dercat}, proving that this result occurs uniformly across stable derivators and it is then independent of coefficients. While we can easily get $\mathrm{D}(kA_n)$ as the underlying category of the shifted derivator $\D^{A_n}$, for $\D$ the derivator associated to the Grothendieck category $\Mod(k)$ (see Example \ref{egs:stable}), the derivator whose underlying category is $\mathrm{D}(kA_n/I)$ is not straightforwardly defined. For this reason we introduced strict full subderivators (\cref{strictfullsubder}). In particular, they allow to enhance in a direct and natural way the relations we have on the quiver so that, in this section, we are able to construct the derivator enhancement of $\mathrm{D}(kA_n/I)$.
\begin{defn}\label{def:A(n,2)}
Let $\D$ be a stable derivator and fix $n \in \N$, we consider the following subposet of $[n-2] \times [n-2]$:
\begin{equation}\label{A(n,2)}
    \adjustbox{scale=0.8}{\begin{tikzcd}
                              &                               &                                                                & { (n-3,n-2)} \arrow[r]                                         & {(n-2,n-2)}           \\
                              &                               & {(2,3)} \arrow[ru, no head, dashed] \arrow[r, no head, dashed] & {} \arrow[u] \arrow[r]                                         & {(n-2,n-3)} \arrow[u] \\
                              & {(1,2)} \arrow[ru] \arrow[rr] &                                                                & {(3,2)} \arrow[ru, no head, dashed] \arrow[u, no head, dashed] &                       \\
{(0,1)} \arrow[ru] \arrow[rr] &                               & {(2,1)} \arrow[ru] \arrow[uu]                                  &                                                                &                       \\
{(0,0)} \arrow[u] \arrow[r]   & {(1,0)} \arrow[ru] \arrow[uu] &                                                                &                                                                &                      
\end{tikzcd}}.
\end{equation}
In particular, we have that every square commutes. 
We call this shape $\tilde{A}(n,2)$ and we denote by $\D^{A(n,2)}$ the strict full subderivator of $\D^{\tilde{A}(n,2)}$ spanned by all coherent diagrams of the shape $\tilde{A}(n,2)$ which vanish at $(i,i+1)$, for $0\leq i \leq n-3$ (cf. \ref{vanish}).
\end{defn}

\begin{proposition}\label{derfield}
Let $\D$ be the stable derivator associated to Grothendieck category $\Mod(k)$, then $\D^{A(n,2)}$ is a derivator enhancement of $\mathrm{D}(kA_n/I)$. Moreover, we have an equivalence of derivators: 
\begin{equation*}
    \D_k^{A(n,2)} \cong \D_{kA_n/I},
    \end{equation*}
where the second derivator is an instance of Example \ref{egs:stable} for the Grothendieck category $\Mod(kA_n/I)$.
\end{proposition}

\begin{proof}
Recall that, by Example \ref{egs:stable}, the derivator of the Grothendieck category $\Mod (kA_n/I)$ is the homotopy derivator associated to the combinatorial model category for complexes over $\Mod(kA_n/I)$. Its underlying category is $\mathrm{D(kA_n/I)}$. In particular, we want to show that $\mathrm{D(kA_n/I)}$ is equivalent to the underlying category of $\D^{A(n,2)}$ which is the full subcategory of $\D(k\tilde{A}(n,2))$ satisfying the vanishing conditions imposed on the strict full subderivator. Namely, the vanishing conditions on $\D^{A(n,2)}$ imply that the complexes of vector spaces at positions $(i,i+1)$ for $0\leq i \leq n-3$ are acyclic in the underlying category. We denote these acyclic complexes by the letters 
 $\mathsf{A}_i$ for $1 \leq i \leq n-2$.
 
\smallskip

Let us now define the natural quasi-isomorphisms which give the desired derived equivalence. Given an object 
 \begin{equation}\label{2}
\begin{tikzcd}
                                                           &                                                            &                                                                     & \mathsf{A}_{n-2} \arrow[r, "z"]                          & Z                 \\
                                                           &                                                            & \mathsf{A}_3 \arrow[ru, no head, dashed] \arrow[r, no head, dashed] & {} \arrow[u] \arrow[r]                                   & Y \arrow[u, "l"'] \\
                                                           & \mathsf{A}_2 \arrow[ru, "v"] \arrow[rr, "r",pos=0.3] &                                                                     & X \arrow[ru, no head, dashed] \arrow[u, no head, dashed] &                   \\
\mathsf{A}_1 \arrow[ru, "u"] \arrow[rr, "p",pos=0.3] &                                                            & W \arrow[ru, "h"'] \arrow[uu, "s"',pos=0.3]                   &                                                          &                   \\
U \arrow[u, "m"] \arrow[r, "f"']                           & V \arrow[ru, "g"'] \arrow[uu, "q"',pos=0.3]          &                                                                     &                                                          &                  
\end{tikzcd}\end{equation}

\noindent in $\mathrm{D}(k ^{A(n,2)})$, via the obvious projections this is quasi-isomorphic to 

\begin{equation}\label{acyclic2}
\begin{tikzcd}
                                                           &                                                                          &                                                                                            & \mathsf{A}_{n-2} \arrow[r, "M_8"]                                                            & Z\oplus\mathsf{A}_{n-2}                                         \\
                                                           &                                                                          & \mathsf{A}_3 \arrow[ru, no head, dashed] \arrow[r, no head, dashed]                        & {} \arrow[u] \arrow[r]                                                                       & Y\oplus\mathsf{A}_{n-3}\oplus\mathsf{A}_{n-2} \arrow[u, "M_9"'] \\
                                                           & \mathsf{A}_2 \arrow[ru, "v"] \arrow[rr, "M_5",pos=0.3]               &                                                                                            & X\oplus\mathsf{A}_2\oplus\mathsf{A}_3 \arrow[ru, no head, dashed] \arrow[u, no head, dashed] &                                                                 \\
\mathsf{A}_1 \arrow[ru, "u"] \arrow[rr, "M_2",pos=0.3] &                                                                          & W\oplus\mathsf{A}_1\oplus\mathsf{A}_2 \arrow[ru, "M_7"'] \arrow[uu, "M_6"',pos=0.3] &                                                                                              &                                                                 \\
U \arrow[u, "m"] \arrow[r, "M_1"']                         & V\oplus\mathsf{A}_1 \arrow[ru, "M_4"'] \arrow[uu, "M_3"',pos=0.3] &                                                                                            &                                                                                              &                                                                
\end{tikzcd}
\end{equation}

\noindent where the matrices $M_1\ldots M_9$ are given by
\[
M_1 = (f \,\, m), \quad M_2 = (p \,\, 1\,\, u)^T, \quad M_3 = (q\,\, 0), \quad M_5 = (r\,\, 1 \,\, v), \quad M_6 = (s \,\, 0), \quad M_8 = (z\,\, 1)^T
\]
as well as by
\[
M_4 = \begin{pmatrix}
g & 0\\
0 & 1\\
q & 0
\end{pmatrix} \qquad M_7=\begin{pmatrix}
h & 0 & r \\
0 & 1 & 1 \\
s & 0 & v
\end{pmatrix} \qquad M_9=\begin{pmatrix}
l & 0 & 0 \\
0 & 1 & 0
\end{pmatrix}.
\]

Observe that \eqref{acyclic2} is also quasi-isomorphic to 

\begin{equation}\label{acyclic3}
\begin{tikzcd}
                                      &                                                                  &                                                                  & 0 \arrow[r]                                                                & Z                                         \\
                                      &                                                                  & 0 \arrow[ru, no head, dashed] \arrow[r, no head, dashed]         & {} \arrow[u] \arrow[r]                                                     & Y\oplus\mathsf{A}_{n-2} \arrow[u, "N_3"'] \\
                                      & 0 \arrow[ru] \arrow[rr]                            &                                                                  & X\oplus\mathsf{A}_3 \arrow[ru, no head, dashed] \arrow[u, no head, dashed] &                                           \\
0 \arrow[ru] \arrow[rr] &                                                                  & W\oplus\mathsf{A}_2 \arrow[ru, "N_2"'] \arrow[uu] &                                                                            &                                           \\
U \arrow[u, "m"] \arrow[r, "M_1"']    & V\oplus\mathsf{A}_1 \arrow[ru, "N_1"'] \arrow[uu] &                                                                  &                                                                            &                                          
\end{tikzcd}
\end{equation}
where
\[
N_1=\begin{pmatrix}
g & -p \\
q & -u
\end{pmatrix}, \qquad N_2=\begin{pmatrix}
h & -r \\
s & -v
\end{pmatrix}, \qquad N_3=\begin{pmatrix}
l & -z 
\end{pmatrix}.
\]

The quasi-isomorphism is given by identities where possible, zero maps to the zero objects and the following morphisms:

\begin{equation}
\begin{tikzcd}[ampersand replacement=\&]
W\oplus\mathsf{A}_1\oplus\mathsf{A}_2 \arrow[dd, "{\begin{pmatrix} 1 & 0 & -p \\ 0 & 1 & -u \end{pmatrix}}"] \&  \& X\oplus\mathsf{A}_2\oplus\mathsf{A}_3 \arrow[dd, "{\begin{pmatrix} 1 & 0 & -r\\ 0 & 1 & -v \end{pmatrix}}"] \&  \& Z\oplus\mathsf{A}_{n-2} \arrow[dd, "{\begin{pmatrix} 1 & -z\end{pmatrix}}"] \\
                                                                                                           \&  \&                                                                                                           \&  \&                                         \\
W\oplus\mathsf{A}_2                                                                                        \&  \& X\oplus\mathsf{A}_3                                                                                       \&  \& Z                                      
\end{tikzcd}
\end{equation}
Let us notice that $\D^{A(n,2)}(e)$ is a full subcategory of $\D(k\tilde{A}(n,2))$ with the components at $(i,i+1)$ acyclic. We now define the equivalence
 \[
\begin{tikzcd}
{\mathrm{D}(kA_n/I)} \arrow[rr, "F", shift left=2] &  & {\D^{A(n,2)}(e),} \arrow[ll, "G", shift left=2]
\end{tikzcd}\]
where $F$ is in the obvious inclusion where we add zero objects in positions $(i,i+1)$ and the functor $G$ sends \eqref{2} to \eqref{acyclic3}. Indeed, we observe that \eqref{acyclic3} can be easily converted to a complex of modules over $kA_n/I$.
Let us prove that these are indeed equivalences. In one direction, the composition $GF$ is the identity functor and in the other direction, there is a natural isomorphism of $FG$ to the identity functor on $\D^{A(n,2)}(e)$ given by the zig-zag of quasi-isomorphisms.\end{proof}

\noindent The structure of strict full subderivators was the key to construct the enhancement for the derived category of $A_n$ with relations given by the ideal $I$. However, whether this structure can enhance a generic quiver with relations is still an open question. Indeed, the most intuitive poset we would draw to enhance $\mathrm{D(kA_n/I)}$ is without the arrows between the acyclic objects, however, they play a fundamental role from the homotopy theory point of view. We now explain better why through an example on $A(4,2)$. Let $\tilde{A}(4,2,-)$ be the following poset
\[
\adjustbox{scale=0.8}{
\begin{tikzcd}
                            & {(1,2)} \arrow[r]             & {(2,2)}           \\
{(0,1)} \arrow[rr]          &                               & {(2,1)} \arrow[u] \\
{(0,0)} \arrow[r] \arrow[u] & {(1,0)} \arrow[ru] \arrow[uu] &                   \\
                            & {\tilde{A}(4,2,-)}            &                  
\end{tikzcd}}\]

\noindent We call $\D^{A(4,2,-)}$ the strict full subderivator of $\D^{\tilde{A}(4,2,-)}$ with same vanishing condition as $\D^{A(4,2)}$.\\
Suppose that the arrow $(0,1)\to(1,2)$ in $\tilde{A}(4,2)$ is not necessary to enhance $\mathrm{D}(kA_n/I)$, then this would imply \[
\D^{A(4,2)} \cong \D^{A(4,2,-)}.\]
However this is not true, for example, when $\D$ is the derivator associated to the Grothendieck category $\Mod(k)$, for a field $k$. Indeed, they have different underlying categories which are not equivalent: as we saw in \eqref{derfield}, $\D^{A(4,2)}(e)\cong\mathrm{D}(kA_n/I)$ which, by \cite{kellerderived}, is equivalent to $\mathrm{D}(kA_n)$. In \cref{cor:derfield2} we prove that $\D^{A(4,2,-)}(e)$ is equivalent to $\mathrm{D}(kD_4)$ where $D_4$ is the Dynkin quiver:
\[
\begin{tikzcd}
            & 2                     \\
0 \arrow[r] & 1 \arrow[u] \arrow[d] \\
            & 3                    
\end{tikzcd}.\]

Since in \cite{kellerderived} it is proved that $\mathrm{D}(kD_4)$ is not equivalent to $\mathrm{D}(kA_n)$, then $\D^{A(4,2)}$ can't be equivalent to $\D^{A(4,2,-)}$.

\section{Main theorem and proof}\label{sec5}

In this section we state and prove the main result of this article.

Observe that the free category of $A_n$ is isomorphic to $[n-1]$.

\begin{theorem}\label{mainthm}
Let $\D$ be a stable derivator. Then, for any integer $n\ge 3$, there exists an equivalence of stable derivators \begin{equation}\label{maineq}
    \begin{tikzcd}
{\D^{A(n,2)}} \arrow[rr, "i^n", shift left=2] &  & \D^{A_n} \arrow[ll, "G^n", shift left=2]
\end{tikzcd}\end{equation} which is natural with respect to exact morphisms.
\end{theorem}
\begin{proof} The proof is divided in $7$ steps and is based only on Kan extensions and restrictions of inclusion functors between posets which, for the case $n=4$, are all illustrated in \cref{$A_4$}. We prove this result by induction, starting from the case $n=3$.

\smallskip

\textsf{First step: Case $n=3$.}\\
If $n=3$, we consider $ \tilde{A}(3,2)$ as a poset of the shape $\square=[1]\times[1]$. The goal is to construct a chain of equivalences whose composition gives us the desired one. Since all the functors considered in this step are Kan extensions of fully faithful functors, thanks to \cref{kanextful}, it is enough to check their essential images. 
Following the total cofiber construction (Construction \ref{tcof}), we define the inclusion of the target square functor \[
\begin{tikzcd}
{(0,1)} \arrow[r]           & {(2,1)}           \\
{(0,0)} \arrow[u] \arrow[r] & {(1,0)} \arrow[u] \\
                            & {\tilde{A}(3,2)} 
\end{tikzcd}\quad\quad\stackrel{i^3_{1,1}}{\longrightarrow}\quad\quad
\begin{tikzcd}
{(0,1)} \arrow[r]           & {(1,1)} \arrow[r]   & {(2,1)} \\
{(0,0)} \arrow[u] \arrow[r] & {(1,0)} \arrow[u]   &         \\
                            & {\tilde{K}^3_{1,2}} &        
\end{tikzcd}.\]
Recalling the conditions on $\D^{A(n,2)}$, by Proposition \ref{prop:cocart-cocone}, an element $X \in \D^{\tilde{K}^3_{1,2}}$ belongs to the essential image of \[(i^3_{1,1})_!:\D^{A(n,2)} \to \D^{\tilde{K}^3_{1,2}}\] if and only if the square
\begin{equation}
    \label{square1}
\begin{tikzcd}
X_{(0,1)} \arrow[r]           & X_{(1,1)}           \\
X_{(0,0)} \arrow[u] \arrow[r] & X_{(1,0)} \arrow[u]
\end{tikzcd}\end{equation}
is cocartesian and $X_{(0,1)}=0$. Since $\D$ is stable, the square (\ref{square1}) is also cartesian. 
We denote by $\D^{K^3_{1,2}}$ this essential image. By the characterization we have just given, $\D^{K^3_{1,2}}$ is a strict full subderivator of $\D^{\tilde{K}^3_{1,2}}$. Since $X_{(0,1)}=0$, we observe that $X_{(1,1)}$ is the cone of the map \[X_{(0,0)} \to  X_{(1,0)}.\]
Consider now a new inclusion of posets 
\[\begin{tikzcd}
{(1,1)} \arrow[r]    & (2,1) \\
{(1,0)} \arrow[u]    &        \\
A_3 &       
\end{tikzcd}\quad\quad\stackrel{i^3_{1,3}}{\longrightarrow}\quad\quad
\begin{tikzcd}
{(0,1)} \arrow[r] & {(1,1)} \arrow[r]   & (2,1) \\
                  & {(1,0)} \arrow[u]   &        \\
                  & {\tilde{K}^3_{1,3}} &       
\end{tikzcd}.\]
It is a cosieve, hence, by Proposition \ref{extbyzero}, \[(i^3_{1,3})_!\colon \D^{A_3} \to \D^{\tilde{K}^3_{1,3}}\] is the left extension by zero. Thus $(i^3_{1,3})_!$ induces an equivalence onto the strict full subderivator \[
\D^{K^3_{1,3}} \subseteq \D^{\tilde{K}^3_{1,3}}\]
spanned by all diagrams which vanish at $(0,1)$.
Finally, we include $\tilde{K}^3_{1,3}$ in $\tilde{K}^3_{1,2}$ through the map

\[
\begin{tikzcd}
{(0,1)} \arrow[r] & {(1,1)} \arrow[r]   & (2,1) \\
                  & {(1,0)} \arrow[u]   &        \\
                  & {\tilde{K}^3_{1,3}} &       
\end{tikzcd}\quad\quad\stackrel{i^3_{1,2}}{\longrightarrow}\quad\quad \begin{tikzcd}
{(0,1)} \arrow[r]           & {(1,1)} \arrow[r]   & {(2,1)} \\
{(0,0)} \arrow[u] \arrow[r] & {(1,0)} \arrow[u]   &         \\
                            & {\tilde{K}^3_{1,2}} &        
\end{tikzcd}.
\]
By Proposition \ref{ccsq}, the right Kan extension \[(i^3_{1,2})_*:\D^{K^3_{1,3}} \to \D^{\tilde{K}^3_{1,2}}\] induces an equivalence onto the essential image that consists of the objects $X \in \D^{\tilde{K}^3_{1,2}}$ such that the square (\ref{square1})
is cartesian and vanishes in the position $(0,1)$ \ie  it induces an equivalence onto $\D^{K^3_{1,2}}$.
Since both the Kan extensions $(i^3_{1,2})_*$ and $(i^3_{1,3})_!$ are fully faithful, the composition  $(i^3_{1,2})_*(i^3_{1,3})_!$ is fully faithful and then it induces an equivalence onto the strict full subderivator $\D^{K^3_{1,2}}$ so that we can consider the inverse equivalence \[((i^3_{1,2})_*(i^3_{1,3})_!)^{-1}=(i^3_{1,3})_!^{-1}(i^3_{1,2})_*^{-1}: \D^{K^3_{1,2}} \to \D^{A_3}.\]
Notice that in this case the restrictions do restrict to functors between the corresponding strict full subderivators. Thus, we get the following simplified formula:
\[(i^3_{1,3})_!^{-1}(i^3_{1,2})_*^{-1}=(i^3_{1,3})^*(i^3_{1,2})^*.\]

Since we also have the equivalence given by $(i^3_{1,1})_!$, we get the desired equivalence by considering the following composition:
\[\D^{A(3,2)} \stackrel{(i^3_{1,1})_!}{\longrightarrow}\D^{K^3_{1,2}}\stackrel{(i^3_{1,2})^*}{\longrightarrow}\D^{K^3_{1,3}}\stackrel{(i^3_{1,3})^*}{\longrightarrow}\D^{A_3}.\]
From now on, let us denote by $\tilde{K}^3_{1,4}$ the coherent diagram of shape  $A_3$.
\oneline
\textsf{Second step: $\D^{A(n,2)} \cong \D^{K^n_{1,2}}\cong\D^{K^n_{1,3}}\cong\D^{K^n_{1,4}}$ for $n\ge 4$.}

In this step we consider the generic case $\D^{A(n,2)}$ and we explain the first part of the inductive passage to reduce the problem to the case $\D^{A(n-1,2)}$.  We apply the same strategy as in the \textit{First Step} in order to ``delete''  the first object subject to the vanishing condition in $\D^{A(n,2)}$. After this procedure we get a poset with only $n-3$ objects subject to the vanishing conditions as in $\D^{A(n-1,2)}$.

We consider the following posets and inclusions, for $n\ge 4$.
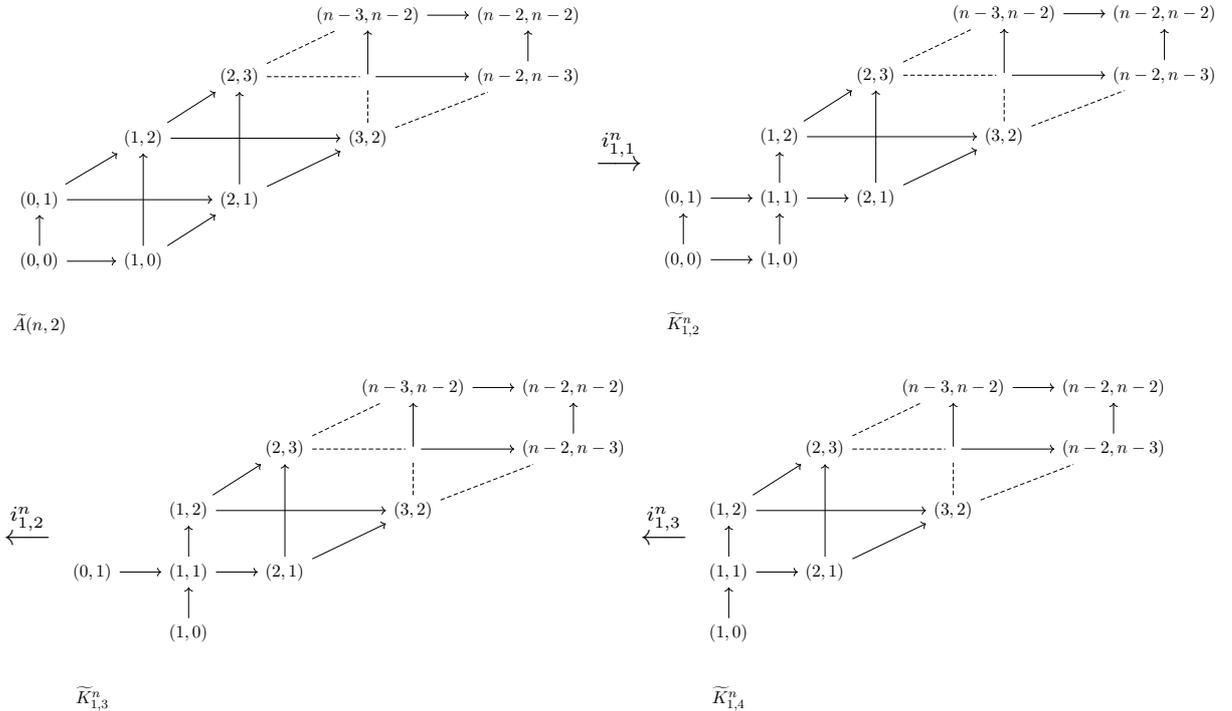
\[\adjustbox{scale=0.6}{
\begin{tikzcd}
                              &                               &                                                                & { (n-3,n-2)} \arrow[r]                                         & {(n-2,n-2)}           \\
                              &                               & {(2,3)} \arrow[ru, no head, dashed] \arrow[r, no head, dashed] & {} \arrow[u] \arrow[r]                                         & {(n-2,n-3)} \arrow[u] \\
                              & {(1,2)} \arrow[ru] \arrow[rr] &                                                                & {(3,2)} \arrow[ru, no head, dashed] \arrow[u, no head, dashed] &                       \\
{(0,1)} \arrow[ru] \arrow[rr] &                               & {(2,1)} \arrow[ru] \arrow[uu]                                  &                                                                &                       \\
{(0,0)} \arrow[u] \arrow[r]   & {(1,0)} \arrow[ru] \arrow[uu] &                                                                &                                                                &                       \\
{\tilde{A}(n,2)}                      &                               &                                                                &                                                                &                      
\end{tikzcd}} \stackrel{i^n_{1,1}}{\longrightarrow}\adjustbox{scale=0.6}{
\begin{tikzcd}
                            &                               &                                                                & { (n-3,n-2)} \arrow[r]                                         & {(n-2,n-2)}           \\
                            &                               & {(2,3)} \arrow[ru, no head, dashed] \arrow[r, no head, dashed] & {} \arrow[u] \arrow[r]                                         & {(n-2,n-3)} \arrow[u] \\
                            & {(1,2)} \arrow[ru] \arrow[rr] &                                                                & {(3,2)} \arrow[ru, no head, dashed] \arrow[u, no head, dashed] &                       \\
{(0,1)} \arrow[r]           & {(1,1)} \arrow[u] \arrow[r]   & {(2,1)} \arrow[ru] \arrow[uu]                                  &                                                                &                       \\
{(0,0)} \arrow[u] \arrow[r] & {(1,0)} \arrow[u]             &                                                                &                                                                &                       \\
{\tilde{K}^n_{1,2}}         &                               &                                                                &                                                                &                      
\end{tikzcd}}
\]

\[\stackrel{i^n_{1,2}}{\longleftarrow}\adjustbox{scale=0.6}{
\begin{tikzcd}
                    &                               &                                                                & { (n-3,n-2)} \arrow[r]                                         & {(n-2,n-2)}           \\
                    &                               & {(2,3)} \arrow[ru, no head, dashed] \arrow[r, no head, dashed] & {} \arrow[u] \arrow[r]                                         & {(n-2,n-3)} \arrow[u] \\
                    & {(1,2)} \arrow[ru] \arrow[rr] &                                                                & {(3,2)} \arrow[ru, no head, dashed] \arrow[u, no head, dashed] &                       \\
{(0,1)} \arrow[r]   & {(1,1)} \arrow[u] \arrow[r]   & {(2,1)} \arrow[ru] \arrow[uu]                                  &                                                                &                       \\
                    & {(1,0)} \arrow[u]             &                                                                &                                                                &                       \\
{\tilde{K}^n_{1,3}} &                               &                                                                &                                                                &                      
\end{tikzcd}}\stackrel{i^n_{1,3}}{\longleftarrow} \adjustbox{scale=0.6}{
\begin{tikzcd}
                              &                                                                & { (n-3,n-2)} \arrow[r]                                         & {(n-2,n-2)}           \\
                              & {(2,3)} \arrow[ru, no head, dashed] \arrow[r, no head, dashed] & {} \arrow[u] \arrow[r]                                         & {(n-2,n-3)} \arrow[u] \\
{(1,2)} \arrow[ru] \arrow[rr] &                                                                & {(3,2)} \arrow[ru, no head, dashed] \arrow[u, no head, dashed] &                       \\
{(1,1)} \arrow[u] \arrow[r]   & {(2,1)} \arrow[ru] \arrow[uu]                                  &                                                                &                       \\
{(1,0)} \arrow[u]             &                                                                &                                                                &                       \\
{\tilde{K}^n_{1,4}}                   &                                                                &                                                                &                      
\end{tikzcd}}
\]
Our aim is again to construct a chain of fully faithful functors whose composition induces an equivalence onto the essential image. By \cref{kanextful}, all the functors we consider in this step are fully faithful, thus it suffices to check their essential images.
 
By applying \cref{prop:intersection}, let us define \[
\D^{\hat{K}^n_{1,m}} \subseteq \D^{\tilde{K}^n_{1,m}}, \quad m = 2,3,4\]
as the strict full subderivator 
spanned by all the coherent diagrams which vanish at $(i,i+1)$ for \[\begin{cases}
  0\leq i \leq n-3  &\text{if\quad} m=2,3\\
  1\leq i \leq n-3  &\text{if\quad} m=4.
\end{cases}\] 
Consider the inclusion $i^n_{1,1}$, by Proposition \ref{ccsq}, the essential image of the left Kan extension \[(i^n_{1,1})_!:\D^{A(n,2)}\to \D^{\tilde{K}^n_{1,2}}\] consists of the objects $X \in \D^{\hat{K}^n_{1,2}}$ such that the square (\ref{square1}) 
is cocartesian. We denote by $\D^{K^n_{1,2}}$ this essential image. As in the previous step, we observe that $\D^{K^n_{1,2}}$ is a strict full subderivator of $\D^{\hat{K}^n_{1,2}}$ and $X_{(1,1)}$ is the cone of the map \[X_{(0,0)} \to  X_{(1,0)}.\]
Consider now the map  $i^n_{1,3}$, it is the inclusion of a cosieve. Hence it follows from Proposition \ref{extbyzero} that \[(i^n_{1,3})_!:\D^{\hat{K}^n_{1,4}} \to \D^{\tilde{K}^n_{1,3}}\] is the left extension by zero. The essential image of this functor is $\D^{\hat{K}^n_{1,3}}$ and for consistence we denote $\D^{K^n_{1,3}}=\D^{\hat{K}^n_{1,3}}$.

\noindent Next, let us observe that $\D^{K^n_{1,2}}$ coincides with the essential image of the right Kan extension of $i^n_{1,2}$
\[(i^n_{1,2})_*:\D^{K^n_{1,3}} \to \D^{\tilde{K}^n_{1,2}}.\]
Indeed, by Proposition \ref{ccsq}, it consists of the objects $X \in \D^{\hat{K}^n_{1,2}}$ such that the square (\ref{square1})
is cartesian. By an analogous reasoning as in the previous step, this implies that the essential image of the functor $(i^n_{1,2})_*(i^n_{1,3})_!$ coincides with the essential image of $(i^n_{1,1})_!$. Thus, the desired equivalence is obtained as the following composition of functors
\[\D^{A(n,2)} \stackrel{(i^n_{1,1})_!}{\longrightarrow}\D^{K^n_{1,2}}\stackrel{(i^n_{1,2})^*}{\longrightarrow}\D^{K^n_{1,3}}\stackrel{(i^n_{1,3})^*}{\longrightarrow}\D^{\hat{K}^n_{1,4}}=\D^{K^n_{1,4}}.\]
Here the last equality holds since $(i^n_{1,3})_!$ is an equivalence and since $\D^{K^n_{1,4}}=\D^{\hat{K}^n_{1,4}}$ by definition.
\oneline
\textsf{Third step: Definition of $\D^{\hat{K}^n_{l,m}}$ for $2 \leq l \leq n-2$,\, $1\leq m \leq 4$, $n\ge4$.}\\
The construction in the previous step leads to the derivator $\D^{K^n_{1,4}}$ with only $n-3$ objects subject to the vanishing conditions. Still, $\D^{K^n_{1,4}} \ne \D^{A(n-1,2)}$; in particular, the difference between them is the arrow $(1,0) \to (1,1)$. The idea is then to manage this arrow by ``bending'' it and construct a $3$-dimensional poset where, by fixing the third coordinate, we find the poset $\tilde{A}(n-1,2)$ (see \textit{Fourth Step}). The resulting new posets are denoted by \[\tilde{K}^n_{l,m}\]
where $n$ comes from $A(n,2)$, $1\leq m \leq 4$ indicates the passages of the procedure in the \textit{Second Step} and $2\leq l \leq n-2$ is the index for the new third dimension: in particular, $l-1$ equals the number of arrows we ``bent''. As posets \[\tilde{K}^n_{l,m}=\tilde{K}^{n-l+1}_{1,m}\times [l-1]\]for $2\leq l \leq  n-2$, $1\leq m \leq 4$ with the componentwise order where we denote
\[
\tilde{K}^{n-l+1}_{1,1} = \tilde{A}(n-l+1,2)
\]
as posets.
We then define the strict full subderivators of $\D^{\tilde{K}^n_{l,m}}$ we need. 

\smallskip

\noindent Consider the poset of the shape \[\tilde{K}^{n-l+1}_{1,m}\times [l-1]\]
for $2\leq l \leq n-2$, $1\leq m \leq 4$ where $\tilde{K}^{n-l+1}_{1,1}$ denotes the poset $\tilde{A}(n-l+1,2)$.

\noindent Thanks to \cref{prop:intersection}, we define \[
\D^{\hat{K}^{n}_{l,m}} \subseteq \D^{\tilde{K}^{n-l+1}_{1,m}\times [l-1]}
\]
to be the strict full subderivator
spanned by all coherent diagrams which vanish at
\[(x,x+1,y) \text{ \quad for\quad } 1\leq x \leq n-l-2,\quad 0\leq y \leq l-1 \text{ \quad if \quad  } m \ne 4\]\[(x,x+1,y) \text{ \quad for\quad } 0\leq x \leq n-l-2,\quad 0\leq y \leq l-1 \text{ \quad if\quad  } m=4\]and
in addition to this condition we require the arrows
\[ (x,x-1,y) \rightarrow (x,x-1,y+1)\]
\[ (x-1,x,y) \rightarrow (x-1,x,y+1)\]
\[ (n-l-1,n-l-1,y) \rightarrow (n-l-1,n-l-1,y+1)\]
to be isomorphisms for $1\leq x \leq n-l-1$, $0\leq y \leq l-2$. Namely, we want $\D^{\hat{K}^{n}_{l,m}}$ to be the intersection between the strict full subderivator satisfying the vanishing conditions described above and the one satisfying the isomorphism conditions described above. 

\noindent Some of the posets and the exactness conditions are depicted below in \cref{$A_4$}.
\oneline
\textsf{Fourth step: $\D^{K^n_{1,4}} \cong \D^{\hat{K}^n_{2,1}}$ for $n\ge 4$.}\\
In this step we illustrate how to formally pass from the $2$-dimensional poset $\tilde{K}^n_{1,4}$ to the $3$-dimensional poset $\tilde{K}^n_{2,1}$ and we prove the equivalence between the strict full subderivators $\D^{K^n_{1,4}}$ and  $\D^{\hat{K}^n_{2,1}}$ by defining an homotopical epimorphism (\cref{sec:homoepi}). Consider the epimorphism given by  
\begin{align*}
i^n_{1,4} \colon \tilde{K}^n_{2,1} & \to \tilde{K}^n_{1,4}; \\
(x,y,z) & \mapsto (x+1,y+1),\, \text{ if } (x,y,z) \neq 0;\\
(0,0,0) & \mapsto (1,0).
\end{align*}
 We want to prove that it is an homotopical epimorphism and in particular that \[(i^n_{1,4})^*:\D^{K^n_{1,4}}\to \D^{\hat{K}^n_{2,1}}\] is an equivalence. In order to apply Proposition \ref{prop:detect} we have to show that $\essIm((i^n_{1,4})^*) \subseteq \D^{\hat{K}^n_{2,1}}$ and that the unit \[\eta\colon Y \to (i^n_{1,4})^*(i^n_{1,4})_!Y\] is an isomorphism for every $Y \in \D^{\hat{K}^n_{2,1}}$. Clearly $i^n_{1,4}$ is surjective on objects and the inclusion \[
 \essIm((i^n_{1,4})^*)\subseteq \D^{\hat{K}^n_{2,1}}\]
 holds. To show that the unit is an isomorphism, by (Der$2$), it suffices to check the invertibility of $\eta$ at every object $(x,y,z) \in \tilde{K}^n_{2,1}$ \ie  to check that the map 
 \[(x,y,z)^*\eta\colon (x,y,z)^*Y \to (x,y,z)^*(i^n_{1,4})^*(i^n_{1,4})_!Y\]
 is an isomorphism for every $(x,y,z)\in \tilde{K}^n_{2,1}$. By Lemma \ref{lem:detect-Kan-to-colim} this is equivalent to proving that the map
 \begin{equation}\label{iso1}
     ((x,y,z), \id_{i^n_{1,4}((x,y,z))})^*p^*Y \to \colim_{(i^n_{1,4}/i^n_{1,4}((x,y,z))}p^*Y
  \end{equation}
 is an isomorphism. In the case $(x,y,z)=(0,0,0)$, it suffices to notice that the slice category consists of only one object. For $(x,y,z)\ne(0,0,0)$, we observe that $((x,y,1), \id_{i^n_{1,4}((x,y,1))})$ is the the terminal object of the slice category \[
 i^n_{1,4}/i^n_{1,4}((x,y,z)).\]
We proceed in two separate cases.
 \begin{itemize}\item[(1)] For $z=1$, by Lemma \ref{lem:detect-final}, (\ref{iso1}) is an isomorphism if and only if 
\[((x,y,z), \id_{i^n_{1,4}((x,y,z))})^*p^*Y \to ((x,y,z), \id_{i^n_{1,4}((x,y,z))})^*p^*Y\]
is an isomorphism. Since this map is the restriction of the identity, this is enough to conclude.
\item[(2)] For $z=0$, by Lemma \ref{lem:detect-final}, (\ref{iso1}) is an isomorphism if and only if 
\[((x,y,z), \id_{i^n_{1,4}((x,y,z))})^*p^*Y \to ((x,y,1), \id_{i^n_{1,4}((x,y,1))})^*p^*Y\]
is an isomorphism. Notice that in the previous step of the proof, more precisely when defining the strict full subderivator $\D^{\hat{K}^n_{2,1}}$, we required this map to be an isomorphism, and so we are done.
\end{itemize}

\textsf{Fifth step: $\D^{K^n_{l,1}} \cong \D^{K^n_{l,2}}\cong \D^{K^n_{l,3}}\cong \D^{K^n_{l,4}}$ for $2\leq l \leq n-2$, $n\ge 4$.}\\
In this step we explain how we apply the procedure in the \textit{Second step} to the cases where $l>1$. By the first and the second step of the proof, for any stable derivator we have the following equivalence for any $n \ge 3$
\[\D^{A(n,2)}=\D^{K^n_{1,1}} \stackrel{(i^n_{1,1})_!}{\longrightarrow}\D^{K^n_{1,2}}\stackrel{(i^n_{1,2})^*}{\longrightarrow}\D^{K^n_{1,3}}\stackrel{(i^n_{1,3})^*}{\longrightarrow}\D^{K^n_{1,4}}.\]
If, as stable derivator, we now consider $\E=\D^{[l-1]}$ then we have the following situation 
\[\begin{tikzcd}
\D^{K^{n-l+1}_{1,1}\times [l-1]} \arrow[rr, "(i^n_{1,1}\times \id)_!"]               &  & \D^{K^{n-l+1}_{1,2}\times [l-1]} \arrow[rr, "(i^n_{1,2}\times \id)^*"]                &  & \D^{K^{n-l+1}_{1,3}\times [l-1]} \arrow[rr, "(i^n_{1,3}\times \id)^*"]                 &  & \D^{K^{n-l+1}_{1,4}\times [l-1]}                 \\
\D^{K^n_{l,1}} \arrow[u, hook] \arrow[rr, "(i^n_{l,1})_!"] &  & \D^{K^n_{l,2}} \arrow[rr, "(i^n_{l,2})^*"] \arrow[u, hook] &  & \D^{K^n_{l,3}} \arrow[rr, "(i^n_{l,3})^*"] \arrow[u, hook] &  & \D^{K^n_{l,4}}. \arrow[u, hook]
\end{tikzcd}\]
Here $\D^{K^{n-l+1}_{1,m}\times [l-1]}=\E^{K^{n-l+1}_{1,m}}$, which is itself a strict full subderivator of \[\E^{\tilde{K}^{n-l+1}_{1,m}}=\D^{\tilde{K}^{n-l+1}_{1,m} \times [l-1]}=\D^{\tilde{K}^{n-l+1}_{l,m}}\]
for $1\leq m \leq 4$.
All the horizontal top arrows are then equivalences and we define $\D^{K^n_{l,m}}$ for any $1\leq m \leq 4$, $2\leq l \leq n-2$ as the strict full subderivator given by the intersection \[\D^{K^{n-l+1}_{1,m}\times [l-1]} \cap \D^{\hat{K}^n_{l,m}} \subseteq \D^{\hat{K}^n_{l,m}}.\] We observe that $\D^{K^n_{2,1}}$ coincides with $\D^{\hat{K}^n_{2,1}}$ by definition. The bottom maps are the restrictions of the top ones to $\D^{K^n_{l,m}}$: they are well defined by Proposition \ref{prop:fullsubder}. We want to show that the bottom maps are still equivalences. These functors are clearly fully faithful then we only  have to check that the essential images coincides with $\D^{K^n_{l,m}}$ for $2\leq m \leq 4$, $2\leq l \leq n-2$. Consider the functors $f_i\colon e\to [l-1]$ for $0\leq i \leq l-1$, which choose the object $i \in [l-1]$ and consider an object $X^m \in \D^{K^n_{l,m}}$ for $2\leq m \leq 4$. By \cref{prodfunct}, the following holds for every $i$:
\[X^m \in \essIm(i^n_{l,m-1})^{\clubsuit}_{\heartsuit} \iff  (\id \times f_i)^*X^m \in \essIm(i^n_{1,m-1})^{\clubsuit}_{\heartsuit}, \qquad \forall \clubsuit \in \{\emptyset,-1,*\}, \, \heartsuit \in \{!,*,\emptyset\}.
\]

By \cref{kanextful}(2), the above double implication is implied by the commutativity of following diagram \[\begin{tikzcd}
 \D^{K^n_{1,m-1}} \arrow[rr, "(i^n_{1,m-1})^{\clubsuit}_{\heartsuit}"]                &  & \D^{K^n_{1,m}}                 \\
 \D^{K^n_{l,m-1}} \arrow[rr, "(i^n_{l,m-1})^{\clubsuit}_{\heartsuit}"'] \arrow[u, "(\id \times f)^*"] &  & \D^{K^n_{l,m}} \arrow[u, "(\id \times f)^*"']
\end{tikzcd}\]
and by the fact that the top arrows in the diagram above are equivalences.

\smallskip

\textsf{Sixth step: $\D^{K^n_{l,4}}\cong \D^{K^n_{l+1,1}} $ for $2\leq l \leq n-3$, $n\ge4$.}\\
Let us explain the homotopical epimorphism in the \textit{Fourth step} for the cases where $l>1$. In particular, we need to define a composition of two homotopical epimorphisms. 
We first define the new poset \[
\tilde{K}^n_{l,5} \subseteq \tilde{K}^n_{l,4}\]
to be the full subposet 
spanned by all the objects different from \[
(1,0,1),\cdots,(1,0,l-1);\]
see \cref{$A_4$} for an illustration. As in the third step we then define 
\[
\D^{K^n_{l,5}}\subseteq \D^{\tilde{K}^n_{l,5}}
\]to be the strict full subderivator 
spanned by all coherent diagrams which vanish at \[(x,x+1,y)\text{ for } 1\leq x \leq n-l-2, \, 0 \leq y \leq l-1. \]
In addition to this condition we require the arrows
\[ (x,x-1,y) \rightarrow (x,x-1,y+1)\]
\[ (n-l-1,n-l-1,y) \rightarrow (n-l-1,n-l-1,y+1)\]
to be isomorphisms for $1\leq x \leq n-l-1$, $0 \leq y \leq l-2$  
and the arrows
\[ (x-1,x,y) \rightarrow (x-1,x,y+1)\]
to be isomorphisms for $2 \leq x \leq n-l-1$, $0\leq y \leq l-2$. 
\noindent We now define an epimorphism 
\begin{align*}
i^n_{l,4}\colon \tilde{K}^n_{l,4} & \to \tilde{K}^n_{l,5} \\
 (x,y,z) & \mapsto (x,y,z), \quad \text{if\quad} (x,y)\ne (1,0) \\
 (1,0,z) & \mapsto (1,0,0).
\end{align*}
 With techniques similar to those used in the fourth step of the proof, it is possible to verify that $i^n_{l,4}$ is an homotopical epimorphism giving an equivalence if we consider the restriction \[(i^n_{l,4})^*:\D^{K^n_{l,5}} \to \D^{K^n_{l,4}}.\] 
The second map we want to consider is
\begin{align*}i^n_{l,5} \colon \tilde{K}^n_{l+1,1} & \to \tilde{K}^n_{l,5} \\\
(x,y,z) & \mapsto (x+1,y+1,z-1), \quad \text{if} \quad z \ne 0 \\
(x,y,0) & \mapsto (x+1,y+1,0), \quad \text{if} \quad (x,y) \ne (0,0)  \\
(0,0,0) & \mapsto (1,0,0).
\end{align*}
Again with techniques similar to those used in fourth step of the proof, it is possible to verify that $i^n_{l,5}$ is an homotopical isomorphism if we consider the restriction
\[(i^n_{l,5})^*:\D^{K^n_{l,5}} \to \D^{K^n_{l+1,1}}.\] Then we have the following desired equivalences
\[\D^{K^n_{l,4}} \stackrel{((i^n_{l,4})^*)^{-1}}{\longrightarrow}\D^{K^n_{l,5}} \stackrel{(i^n_{l,5})^*}{\longrightarrow} \D^{K^n_{l+1,1}}.\]

\smallskip

\textsf{Seventh step: $\D^{K^n_{n-2,4}}\cong \D^{A_n}$ for $n\ge4$.}\\

Let us define the homotopical epimorphisms which give the equivalence to $\D^{A_n}$. By definition, the inclusion  \[\D^{K^n_{n-2,4}}\subset \D^{\tilde{K}^{n-n+3}_{1,4}\times [n-2-1]}=\D^{\tilde{K}^{3}_{1,4}\times [n-3]}=\D^{A_3\times [n-3]}\]
holds.  For $n=4$, the underlying posets are depicted in \cref{$A_4$}. We recall that the only conditions we have on the strict full subderivator $\D^{K^n_{n-2,4}}$ are the isomorphism conditions on the arrows
\[ (x,x-1,y) \rightarrow (x,x-1,y+1)\]
\[ (n-l-1,n-l-1,y) \rightarrow (n-l-1,n-l-1,y+1)\]
for $1 \leq x \leq n-l-1$ and $0 \leq y \leq l-2$. 

\smallskip

Similarly as in the previous step, the strategy of the proof consists in defining maps between posets that turn out to be equivalences given by homotopical epimorphisms. 
Analogously as before, let us construct the equivalence
\[\D^{K^n_{n-2,4}} \stackrel{((i^n_{n-2,4})^*)^{-1}}{\longrightarrow}\D^{K^n_{n-2,5}}. \]

\noindent We now define the following new map
\begin{align*}
i^n_{n-2,5}:\tilde{K}^n_{n-2,5} & \to A_n \\\
(1,1,z) & \mapsto z+2  \\\
(1,0,z) & \mapsto 1  \\\
(2,1,z) & \mapsto n.
\end{align*}

\noindent Using similar techniques to those used in the fourth step of the proof, one shows that \[(i^n_{n-2,5})^*:\D^{A_n} \to \D^{K^n_{n-2,5}}\] is an equivalence given by an homotopical epimorphism. Then, we get the last equivalence by taking the inverse of 
$(i^n_{n-2,5})^*$.
\oneline
\textsf{Conclusion.}\\
We conclude the proof considering the equivalence given by the composition of the ones we built in each step, namely

\begin{equation}\label{i^n}
    ((i^n_{n-2,5})^*)^{-1}((i^n_{n-2,4})^*)^{-1}(i^n_{n-2,3})^*(i^n_{n-2,2})^*(i^n_{n-2,1})_!(i^n_{n-3,5})^*((i^n_{n-3,4})^*)^{-1}(i^n_{n-3,3})^*\cdots\end{equation}
    \begin{equation*}
\cdots (i^n_{3,1})_!(i^n_{2,5})^*((i^n_{2,4})^*)^{-1}(i^n_{2,3})^*(i^n_{2,2})^*(i^n_{2,1})_!(i^n_{1,4})^*(i^n_{1,3})^*(i^n_{1,2})^*(i^n_{1,1})_!.
\end{equation*}
 We call this composition $i^n$.
Since we only extended by zeroes, added cocartesian squares and restricted through homotopical epimorphisms, it is possible to check that this equivalence is natural with respect to exact morphisms.\end{proof}

\begin{example}\label{$A_4$} We illustrate below the posets involved in the proof, in the case $n=4$. The underlined objects are the ones we required to be zero objects and the double line arrows are the isomorphisms.

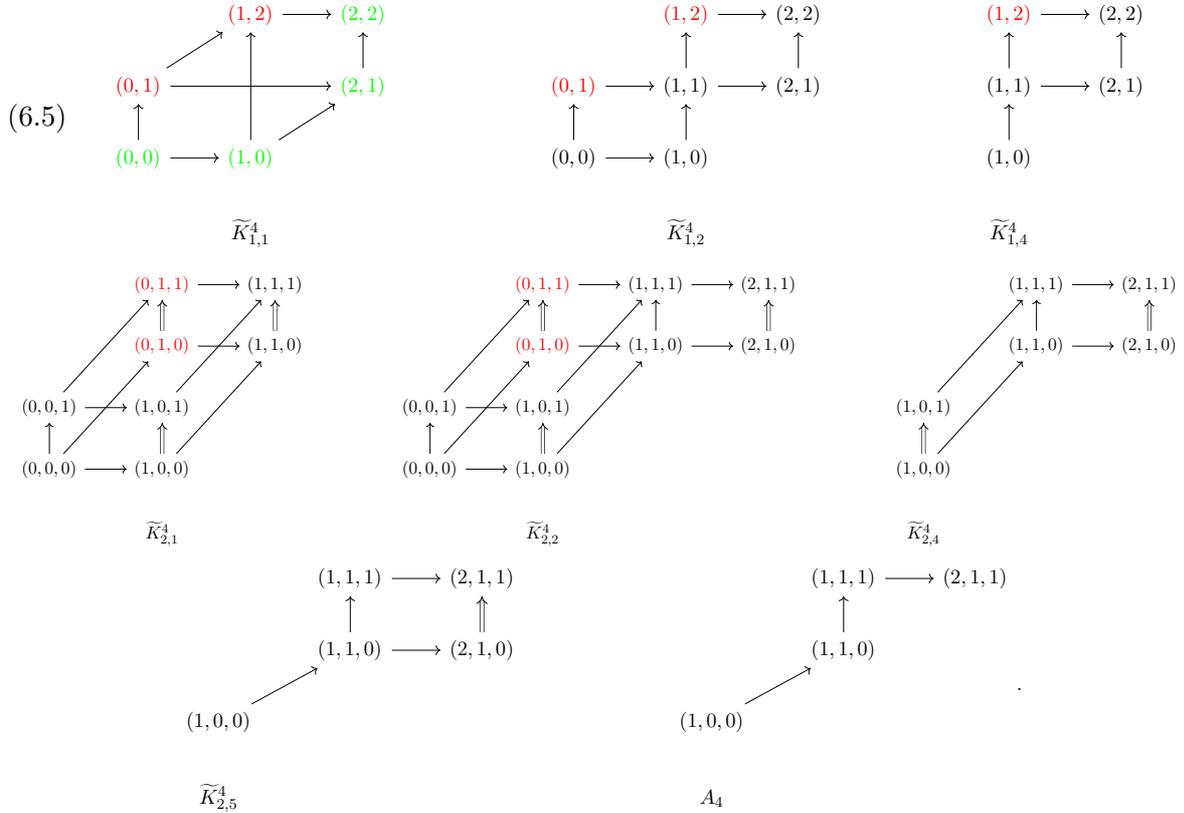
\begin{equation}\label{pic$A(4,2)$}
    \adjustbox{scale=0.7}{
\begin{tikzcd}
                                               & {\underline{(1,2)}} \arrow[r]               & {\textbf{(2,2)}}           &  &  &                                    & {\underline{(1,2)}} \arrow[r] & {(2,2)}           &  &  & {\underline{(1,2)}} \arrow[r] & {(2,2)}           \\
{\underline{(0,1)}} \arrow[ru] \arrow[rr] &                                                  & {\textbf{(2,1)}} \arrow[u] &  &  & {\underline{(0,1)}} \arrow[r] & {(1,1)} \arrow[u] \arrow[r]        & {(2,1)} \arrow[u] &  &  & {(1,1)} \arrow[u] \arrow[r]        & {(2,1)} \arrow[u] \\
{\textbf{(0,0)}} \arrow[u] \arrow[r] & {\textbf{(1,0)}} \arrow[ru] \arrow[uu] &                                      &  &  & {(0,0)} \arrow[u] \arrow[r]        & {(1,0)} \arrow[u]                  &                   &  &  & {(1,0)} \arrow[u]                  &                   \\
                                               & {\tilde{K}^4_{1,1}}                              &                                      &  &  &                                    & {\tilde{K}^4_{1,2}}                &                   &  &  & {\tilde{K}^4_{1,4}}                &                  
\end{tikzcd}}\end{equation}
\[\adjustbox{scale=0.6}{
\begin{tikzcd}
                                          & {\underline{(0,1,1)}} \arrow[r]                       & {(1,1,1)}                       &  &                                           & {\underline{(0,1,1)}} \arrow[r]                       & {(1,1,1)} \arrow[r]           & {(2,1,1)}                       &  &                                             & {(1,1,1)} \arrow[r]           & {(2,1,1)}                       \\
                                          & {\underline{(0,1,0)}} \arrow[u, Rightarrow] \arrow[r] & {(1,1,0)} \arrow[u, Rightarrow] &  &                                           & {\underline{(0,1,0)}} \arrow[r] \arrow[u, Rightarrow] & {(1,1,0)} \arrow[r] \arrow[u] & {(2,1,0)} \arrow[u, Rightarrow] &  &                                             & {(1,1,0)} \arrow[r] \arrow[u] & {(2,1,0)} \arrow[u, Rightarrow] \\
{(0,0,1)} \arrow[r] \arrow[ruu]           & {(1,0,1)} \arrow[ruu]                                      &                                 &  & {(0,0,1)} \arrow[ruu] \arrow[r]           & {(1,0,1)} \arrow[ruu]                                      &                               &                                 &  & {(1,0,1)} \arrow[ruu]                       &                               &                                 \\
{(0,0,0)} \arrow[r] \arrow[u] \arrow[ruu] & {(1,0,0)} \arrow[u, Rightarrow] \arrow[ruu]                &                                 &  & {(0,0,0)} \arrow[r] \arrow[u] \arrow[ruu] & {(1,0,0)} \arrow[u, Rightarrow] \arrow[ruu]                &                               &                                 &  & {(1,0,0)} \arrow[u, Rightarrow] \arrow[ruu] &                               &                                 \\
                                          & {\tilde{K}^4_{2,1}}                                        &                                 &  &                                           & {\tilde{K}^4_{2,2}}                                        &                               &                                 &  & {\tilde{K}^4_{2,4}}                         &                               &                                
\end{tikzcd}}
\]
\[\adjustbox{scale=0.7}{
\begin{tikzcd}
                     & {(1,1,1)} \arrow[r]           & {(2,1,1)}                       &  &  &                      & {(1,1,1)} \arrow[r] & {(2,1,1)} \\
                     & {(1,1,0)} \arrow[u] \arrow[r] & {(2,1,0)} \arrow[u, Rightarrow] &  &  &                      & {(1,1,0)} \arrow[u] &           \\
{(1,0,0)} \arrow[ru] &                               &                                 &  &  & {(1,0,0)} \arrow[ru] &                     &           \\
{\tilde{K}^4_{2,5}}  &                               &                                 &  &  & A_4                  &                     &          
\end{tikzcd}.}
\]
\end{example}

\begin{remark}
 A similar result can be deduced from Falk Beckert's work \cite[Corollary~9.15]{beckert19}, where hypercubes are used to enhance quiver with relations. By composing the equivalence \eqref{maineq} with Beckert's equivalence, we find an equivalence between his hypercubes and $\D^{A(n,2)}$ for any stable derivator $\D$, which is again natural with respect to exact morphisms. For the case $n=3$ the considered hypercube shape coincides with the poset $\tilde{A}(3,2)$. For greater $n$, Beckert gets a $(n-2)$-dimensional hypercube. Here, as proved in \cref{sec:relenhancement}, we propose a more convenient and natural enhancement of quiver with relations which deals only with 2-dimensional posets. Moreover, we give a very elementary proof of \cref{mainthm} which involves only simple Kan extensions and restrictions. 
\end{remark}

We now finish the discussion at the end of \cref{sec:relenhancement} with techniques similar to those used in the proof of \cref{mainthm}.

\begin{proposition}\label{derfield2}
Let $\D$ be a stable derivator, then
$\D^{A(n,2,-)}\cong \D^{D_4}$.
\end{proposition}
\begin{proof}
    Observe that we can also depict the poset $\tilde{A}(4,2,-)$ as follows: 
    \[
\begin{tikzcd}
{(0,1)} \arrow[r]            & {(2,1)} \arrow[r]            & {(2,2)}            \\
                             &                              &                    \\
{(0,0)} \arrow[uu] \arrow[r] & {(1,0)} \arrow[r] \arrow[uu] & {(1,2).} \arrow[uu]
\end{tikzcd}\]

This poset is given by the product $[1]\times[2]$, and, if we require the vanishing conditions on the objects $(0,1),(1,2)$, we get a strict full subderivator of $\D^{[1]\times[2]}$ which is equivalent to $\D^{A(4,2,-)}$. For this reason, with a slight abuse of notation, we also denote it by $\D^{A(4,2,-)}$.
Similarly as in the proof of \cref{mainthm}, the equivalence \[
\D^{A(n,2,-)}\cong \D^{D_4}\]
is then given by Kan extensions and restrictions along the following inclusions. 
As the objects in the posets would not respect the usual ordering, for simplicity, in this proof we change the labels: $O_1$ and $O_2$ are the new labels for the objects on which we require the vanishing conditions. All the inclusions send any object to the one with the same label.

\[
\adjustbox{scale=0.7}{
\begin{tikzcd}
O_1 \arrow[r]               & Z \arrow[r]            & W              &                                       & O_1 \arrow[r] \arrow[rd]    & Z \arrow[r]           & W              &                                      & O_1 \arrow[r] \arrow[rd] & Z \arrow[r]           & W              \\
                            &                        &                & \stackrel{\lambda_1}{\longrightarrow} &                             & cone(f) \arrow[u]     &                & \stackrel{\lambda_2}{\longleftarrow} &                          &  cone(f) \arrow[u]    &                \\
X \arrow[uu] \arrow[r, "f"] & Y \arrow[r] \arrow[uu] & O_2 \arrow[uu] &                                       & X \arrow[uu] \arrow[r, "f"] & Y \arrow[u] \arrow[r] & O_2 \arrow[uu] &                                      &                          & Y \arrow[u] \arrow[r] & O_2 \arrow[uu]
\end{tikzcd}}\]
\vspace{0.5cm}
\[\adjustbox{scale=0.7}{
\begin{tikzcd}
O_1 \arrow[r] \arrow[rd] & Z \arrow[r]           & W              &                                      & Z \arrow[r, "g"]      & W              &                                       & Z \arrow[rr] \arrow[rd]           &        & W                         \\
                         &  cone(f) \arrow[u]    &                & \stackrel{\lambda_3}{\longleftarrow} & cone(f) \arrow[u]     &                & \stackrel{\lambda_4}{\longrightarrow} & cone(f) \arrow[u]                 & fib(g) &                           \\
                         & Y \arrow[u] \arrow[r] & O_2 \arrow[uu] &                                      & Y \arrow[u] \arrow[r] & O_2 \arrow[uu] &                                       & Y \arrow[u] \arrow[rr] \arrow[ru] &        & O_2 \arrow[uu] \arrow[lu]
\end{tikzcd}

}\]
\vspace{0.5cm}
\[\adjustbox{scale=0.7}{
\begin{tikzcd}
Z \arrow[rr]                      &                              & W              &                                      & Z                                 &                              &     &                                      & Z                      &                   \\
cone(f) \arrow[u]                 & fib(g) \arrow[rd] \arrow[lu] &                & \stackrel{\lambda_5}{\longleftarrow} & cone(f) \arrow[u]                 & fib(g) \arrow[lu] \arrow[rd] &     & \stackrel{\lambda_6}{\longleftarrow} & cone(f) \arrow[u]      & fib(g) \arrow[lu] \\
Y \arrow[u] \arrow[rr] \arrow[ru] &                              & O_2 \arrow[uu] &                                      & Y \arrow[u] \arrow[rr] \arrow[ru] &                              & O_2 &                                      & Y \arrow[u] \arrow[ru] &                  
\end{tikzcd}}\]
\vspace{0.5cm}
\[\adjustbox{scale=0.7}{
\begin{tikzcd}
cone(f) \arrow[rr]      &  & Z                 &                                       & cone(f) \arrow[rr]                 &                         & Z                 &                                      & cone(f)                            &                         &        \\
                        &  &                   & \stackrel{\lambda_7}{\longrightarrow} &                                    & Q \arrow[lu] \arrow[rd] &                   & \stackrel{\lambda_8}{\longleftarrow} &                                    & Q \arrow[lu] \arrow[rd] &        \\
Y \arrow[uu] \arrow[rr] &  & fib(g) \arrow[uu] &                                       & Y \arrow[uu] \arrow[rr] \arrow[ru] &                         & fib(g) \arrow[uu] &                                      & Y \arrow[uu] \arrow[rr] \arrow[ru] &                         & fib(g)
\end{tikzcd}
}\]
\vspace{0.5cm}
\[\adjustbox{scale=0.7}{
\begin{tikzcd}
cone(f)                            &                         &        &                                      &             & cone(f)               \\
                                   & Q \arrow[lu] \arrow[rd] &        & \stackrel{}{=} & Y \arrow[r] & Q \arrow[u] \arrow[d] \\
Y \arrow[uu] \arrow[rr] \arrow[ru] &                         & fib(g) &                                      &             & fib(g)               
\end{tikzcd}}.\]
\vspace{0.5cm}

\noindent One can check that the following composition of functors 
\[(\lambda_8)^*(\lambda_7)_*(\lambda_6)^*(\lambda_5)^*(\lambda_4)_*(\lambda_3)^*(\lambda_2)^*(\lambda_1)_!
 \]
 is the one giving the equivalence.
\end{proof}

\begin{corollary}\label{cor:derfield2}
Let $\D$ be a stable derivator associated to the Grothendieck category $\Mod(k)$. Then $\D^{A(n,2,-)}$ is a derivator enhancement of $\mathrm{D}(kD_4)$.
\end{corollary}
\begin{proof}
    The straightforward derivator enhancement of $\mathrm{D}(kD_4)$ is the shifted derivator $\D^{D_4}$. Then the result follows directly from \cref{derfield2}.
\end{proof}

\section{\texorpdfstring{$\infty$}{infinity}-Dold-Kan correspondence via representation theory}\label{sec:bridges}

In this section we aim to explain a connection between representation theory and homotopy theory. Indeed, Theorem \ref{maineq}, which arises as an enhancement of an equivalence in representation theory, is actually also the enhancement of a fundamental result in homotopy theory, the so-called Dold-Kan correspondence. In particular, in the language of derivators, Theorem \ref{maineq} corresponds to a bounded version of the equivalence obtained by Ariotta in \cite[Theorem~4.7]{ariotta} which is a reformulation of the $\infty$-Dold-Kan correspondence. To be able to compare the two statements we use the construction of the coherent Auslander-Reiten quiver developed by Groth and {\v S}{\v t}ov{\'\i}{\v c}ek in \cite{[3]}.  This construction is illustrated in Subsection \ref{reprth}, together with its connection to Theorem \ref{maineq}.
\subsection{Coherent Auslander-Reiten quiver}\label{reprth}
Let us start by briefly recalling some definitions. A \textbf{quiver} $Q$ consists of a set of vertices $Q_0$ and a set of arrows $Q_1$. We can associate to $Q$ the \textbf{repetitive quiver} $\hat{Q}$ whose vertices are pairs $(k,q)$ with $k \in \Z$ and $q \in Q$ and for every arrow $\alpha: q_1 \to q_2$ in $Q_1$ there are arrows $\alpha: (k,q_1) \to (k,q_2)$ and $\tilde{\alpha}: (k,q_2) \to (k+1,q_1)$ in $\hat{Q}_1$. We denote by $M_n$ the category obtained from the repetitive quiver of $A_{n+2}$ by forcing all squares of the form

\begin{equation}
\adjustbox{scale=0.8}{\begin{tikzcd}
                                & {(k,q)} \arrow[rd]   &             \\
{(k,q-1)} \arrow[ru] \arrow[rd] &                      & {(k+1,q-1)} \\
                                & {(k+1,q)} \arrow[ru] &            
\end{tikzcd}}
\end{equation}

to commute.

\begin{example}\label{M_4}
 Below is an illustration of what the category $M_4$ looks like.

\begin{equation}\label{pic$M_4$}
\adjustbox{scale=0.6}{
\begin{tikzcd}
                   &                               &                               &                               &                               & {\underline{(-3,5)}} \arrow[rd]            &                               & {\underline{(-2,5)}} \arrow[rd]            &                               & {\underline{(-1,5)}} \arrow[rd] &                    & {\underline{(0,5)}} \\
                   &                               &                               &                               & {\textbf{(-3,4)}} \arrow[ru] \arrow[rd] &                               & {(-2,4)} \arrow[ru] \arrow[rd] &                               & {(-1,4)} \arrow[ru] \arrow[rd] &                    & {\textbf{(0,4)}} \arrow[ru] &         \\
                   &                               &                               & {(-3,3)} \arrow[ru] \arrow[rd] &                               & {(-2,3)} \arrow[ru] \arrow[rd] &                               & {(-1,3)} \arrow[ru] \arrow[rd] &                               & {(0,3)} \arrow[ru] &                    &         \\
\cdots             &                               & {(-3,2)} \arrow[ru] \arrow[rd] &                               & {(-2,2)} \arrow[ru] \arrow[rd] &                               & {(-1,2)} \arrow[ru] \arrow[rd] &                               & {(0,2)} \arrow[ru]            &                    & \cdots             &         \\
                   & {\textbf{(-3,1)}} \arrow[ru] \arrow[rd] &                               & {(-2,1)} \arrow[ru] \arrow[rd] &                               & {(-1,1)} \arrow[ru] \arrow[rd] &                               & {\textbf{(0,1)}} \arrow[ru]            &                               &                    &                    &         \\
{\underline{(-3,0)}} \arrow[ru] &                               & {\underline{(-2,0)}} \arrow[ru]            &                               & {\underline{(-1,0)}} \arrow[ru]            &                               & {\underline{(0,0)}} \arrow[ru]            &                               &                               &                    &                    &         \\
                   &                               &                               &                               &                               & M_4                           &                               &                               &                               &                    &                    &        
\end{tikzcd}}
\end{equation}
\end{example}
\begin{construction}\label{constr:M_n}
Let $\D$ be a stable derivator. Following \cite[Section.~4]{[3]} we construct a coherent diagram of shape $M_n$ satisfying certain exactness and vanishing conditions. Note that there is a fully faithful functor
\begin{equation}
i_n\colon A_n\to M_n
\label{eq:i}
\end{equation}
\[\quad\quad\quad l\mapsto (0,l)\]
which we consider as an inclusion. This embedding factors as a composition of inclusions of full subcategories
\[i_n\colon A_n\stackrel{s_1}{\to} K_1\stackrel{s_2}{\to} K_2\stackrel{s_3}{\to} K_3\stackrel{s_4}{\to} M_n \]
where
\begin{itemize}
\item[(1)] $K_1$ is obtained from $A_n$ by adding the objects $(k,n+1)$ for $k\geq 0$ and $(k,0)$ for $k>0$,
\item[(2)] $K_2$ contains all objects from $K_1$ and the objects $(k,l)$ for $k>0$, 
\item[(3)] $K_3$ is obtained from $K_2$ by adding the objects $(k,n+1)$ for $k<0$ and $(k,0)$ for $k\leq 0$.
\end{itemize}
The inclusion $s_4$ thus adds the remaining objects in the negative $k$-direction. By \cref{kanextful}, associated to these fully faithful functors there are fully faithful Kan extension functors
\begin{equation}
\vcenter{
\xymatrix{
\D^{A_n}\ar[r]^-{(s_1)_\ast}&\D^{K_1}\ar[r]^-{(s_2)_!}&\D^{K_2}\ar[r]^-{(s_3)_!}&
\D^{K_3}\ar[r]^-{(s_4)_\ast}&\D^{M_n}.
}
}
\label{eq:AR}
\end{equation}

Let us denote by $F_n$ this composition of functors and by 
\[
\D^{M_n,\mathrm{ex}}\subseteq\D^{M_n}\]
the full subderivator spanned by all coherent diagrams which vanish at $(k,0),(k,n+1)$ for all $k\in\Z$ and which make all squares bicartesian. The objects with the vanishing condition are the underlined ones in Figure (\ref{pic$M_4$}). Observe that, by Proposition \ref{prop:fullsubder}, $\D^{M_n,\mathrm{ex}}$ is actually a strict full subderivator and then, in particular, it is a derivator.
\end{construction}
\noindent The above construction brings us to the following theorem.
\begin{theorem}[{\cite[Theorem.~4.6]{[3]}}]\label{thm:AR}
Let $\D$ be a stable derivator. Then \eqref{eq:AR} induces an equivalence of stable derivators \[
\begin{tikzcd}
\D^{A_n} \arrow[rr, "F_n", shift left=2] &  & {\D^{M_n,\mathrm{ex}}} \arrow[ll, "i_n^\ast", shift left=2]
\end{tikzcd}\] which is natural with respect to exact morphisms. Moreover, the inclusion $\D^{M_n,\mathrm{ex}}\to\D^{M_n}$ is exact.
\end{theorem}
\noindent We now describe the relation between our main Theorem \ref{mainthm} and the above Theorem \ref{thm:AR}. In Theorem \ref{mainthm}, we found the equivalence $\D^{A(n,2)}\rightleftarrows \D^{A_n}$ and equivalences between $\D^{A_n}$ and $\D^{K^n_{l,m}}$ for every $1\leq l \leq n-2$, $1\leq m \leq5$.  Then, by composing these equivalences with those in \cref{thm:AR}, we get the following equivalences \[\D^{A(n,2)}\rightleftarrows\D^{M_n,\mathrm{ex}},\quad \D^{K^n_{l,m}}\rightleftarrows\D^{M_n,\mathrm{ex}}\] for every $1\leq l \leq n-2$, $1\leq m \leq5$. 
\begin{example}\label{$n=4$}
Let us explicitly describe these equivalences for $n=4$; one can then deduce the general case from this particular one. Consider the following functors:
\begin{itemize}\item  $j^4:A(4,2)=\tilde{K}^4_{1,1}\to M_4,$ 
\[(0,0) \mapsto (-3,1) \quad (1,0) \mapsto (-3,4)\quad (0,1) \mapsto (-2,0)\]\[(1,2) \mapsto (-2,5) \quad (2,1) \mapsto (0,1) \quad\quad (2,2) \mapsto (0,4).\]
    Except for the objects we required to be zeroes, we show this map through the bold objects in Pictures (\ref{pic$A(4,2)$}) and (\ref{pic$M_4$}).
\oneline
\item $j^4_{1,2}:\tilde{K}^4_{1,2}\to M_4,$ \[(1,1)\mapsto (-2,3)\]
 and on the remaining objects it acts as the inclusion $j^4$ described above.
  \oneline
\item $j^4_{1,3}:\tilde{K}^4_{1,3}\to M_4$ consists of the inclusion $j^4_{1,2}$ restricted to the full subposet $\tilde{K}^4_{1,3} \subset \tilde{K}^4_{1,2} $.
  \oneline
\item $j^4_{1,4}:\tilde{K}^4_{1,4}\to M_4$ consists of the inclusion $j^4_{1,2}$ restricted to the full subposet $\tilde{K}^4_{1,4} \subset \tilde{K}^4_{1,2} $.
\oneline
\item $j^4_{2,1}:\tilde{K}^4_{2,1}\to M_4,$
\[\quad (0,0,0) \mapsto (-3,4)\quad (0,0,1) \mapsto (-2,3) \quad(0,1,0) \mapsto (-3,5) \quad (0,1,1) \mapsto (-2,5) \]\[(1,0,0),(1,0,1) \mapsto (0,1) \quad (1,1,0),(1,1,1) \mapsto (0,4).\]
\oneline
\item $j^4_{2,2}:\tilde{K}^4_{2,2}\to M_4,$ \[\quad (0,0,0) \mapsto (-3,4)\quad (0,0,1) \mapsto (-2,3) \quad(0,1,0) \mapsto (-3,5)\quad (0,1,1) \mapsto (-2,5)\]\[(1,1,0) \mapsto (0,2) \quad(1,0,0),(1,0,1) \mapsto (0,1)\]\[\,\,(1,1,1) \mapsto (0,3) \quad (2,1,0),(2,1,1) \mapsto (0,4).\]
\oneline
We can depict this map as follows
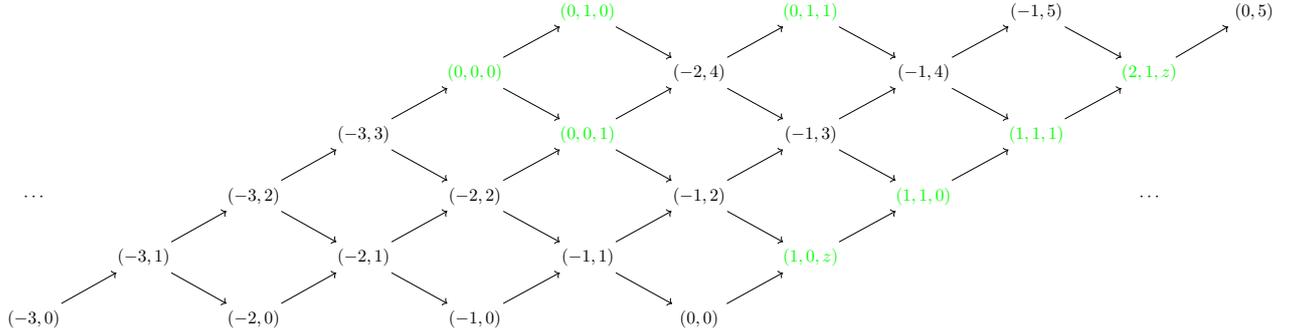
\[\adjustbox{scale=0.6}{
\begin{tikzcd}
                   &                               &                               &                               &                                                   & {\textbf{(0,1,0)}} \arrow[rd]            &                               & {\textbf{(0,1,1)}} \arrow[rd]           &                                        & {(-1,5)} \arrow[rd]                     &                                                 & {(0,5)} \\
                   &                               &                               &                               & {\textbf{(0,0,0)}} \arrow[ru] \arrow[rd] &                                                   & {(-2,4)} \arrow[ru] \arrow[rd] &                                                  & {(-1,4)} \arrow[ru] \arrow[rd]          &                                        & {\textbf{(2,1,z)}} \arrow[ru] &         \\
                   &                               &                               & {(-3,3)} \arrow[ru] \arrow[rd] &                                                   & {\textbf{(0,0,1)}} \arrow[ru] \arrow[rd] &                               & {(-1,3)} \arrow[ru] \arrow[rd]                    &                                        & {\textbf{(1,1,1)}} \arrow[ru] &                                                 &         \\
\cdots             &                               & {(-3,2)} \arrow[ru] \arrow[rd] &                               & {(-2,2)} \arrow[ru] \arrow[rd]                     &                                                   & {(-1,2)} \arrow[ru] \arrow[rd] &                                                  & {\textbf{(1,1,0)}} \arrow[ru] &                                        & \cdots                                          &         \\
                   & {(-3,1)} \arrow[ru] \arrow[rd] &                               & {(-2,1)} \arrow[ru] \arrow[rd] &                                                   & {(-1,1)} \arrow[ru] \arrow[rd]                     &                               & {\textbf{(1,0,z)}} \arrow[ru] &                                        &                                        &                                                 &         \\
{(-3,0)} \arrow[ru] &                               & {(-2,0)} \arrow[ru]            &                               & {(-1,0)} \arrow[ru]                                &                                                   & {(0,0)} \arrow[ru]            &                                                  &                                        &                                        &                                                 &        
\end{tikzcd}}\]
for $z=0,1$.
\oneline
\item $j^4_{2,4}:\tilde{K}^4_{2,4}\to M_4,\quad j^4_{2,5}:\tilde{K}^4_{2,5}\to M_4,\quad i_4:A_4\to M_4.$
They are the restrictions of $j^4_{2,2}$ to the full subposets $\tilde{K}^4_{2,4}$, $\tilde{K}^4_{2,5}$ and $A_4$ respectively. Note that $i_4$ is the map we have already defined (\ref{eq:i}).
\end{itemize} 

\noindent Let us now consider the restrictions along the functors from \cref{$n=4$}. Notice that they commute with the horizontal equivalences in the next diagram. These equivalences are those from the proof of \cref{mainthm} and they are given by restrictions along the corresponding functors in one of the directions. By \cref{thm:AR}, $(i_4)^*$ is an equivalence so also all the other slanted functors in the diagram are equivalences as well.
 
\[\begin{tikzcd}
{\D^{A(4,2)}} \arrow[r, "\sim"] & {\D^{K^4_{1,2}}} \arrow[r, "\sim"] & {\D^{K^4_{1,4}}} & \cdots                                                                                                                                                                                                                                                                      & {\D^{K^4_{2,4}}} \arrow[r, "\sim"] & {\D^{K^4_{2,4}}} \arrow[r, "\sim"] & \D^{A_4} \\
                                  &                                       &                  & \cdots                                                                                                                                                                                                                                                                      &                                      &                                      &                           \\
                                  &                                       &                  & {\D^{M_4, \mathrm{ex}}} \arrow[llluu, "(j^4)^*" description] \arrow[luu, "{(j^4_{1,4})^*}" description] \arrow[ruu, "{(j^4_{2,4})^*}" description] \arrow[rrruu, "(i_4)^*" description] \arrow[lluu, "{(j^4_{1,2})^*}" description] \arrow[rruu, "{(j^4_{2,5})^*}" description] &                                      &                                      &                          
\end{tikzcd}.\]
\end{example}
\begin{construction}\label{cons:jn}
To generalize Example \ref{$n=4$} above, consider the following inclusion for $n \in \N$:
\begin{itemize} 
    \item If $n$ even, $j^n: \tilde{A}(n,2) \to M_n$
    \[ (0,0)\mapsto (-((n-2)/2)(n-1),1), \quad (n,n) \mapsto (0,n),\]
     \[\quad(p,p-1) \mapsto\begin{cases}
  (((-n+p+1)/2)(n-1),n) &\text{if\quad} p\text{\quad odd}\\
  (((-n+p+2)/2)(n-1),1)  &\text{if\quad} p>0,\text{\quad even},
\end{cases}\]
\[\quad\qquad(p-1,p) \mapsto\begin{cases}
  (((-n+p+1)/2)(n-1)+1,0) &\text{if\quad} p\text{\quad odd}\\
  (((-n+p)/2)(n-1)+1,n+1)  &\text{if\quad} p>0,\text{\quad even}.
\end{cases}\]
    \item If $n$ odd, $j^n: A(n,2) \to M_n$
     \[ (0,0)\mapsto (((-n+1)/2)(n-1),n), \quad (n,n) \mapsto (0,n),\]
     \[\quad(p,p-1) \mapsto\begin{cases}
  (((-n+p+2)/2)(n-1),1) &\text{if\quad} p\text{\quad odd}\\
  (((-n+p+1)/2)(n-1),n)  &\text{if\quad} p>0,\text{\quad even},
\end{cases}\]
\[\quad\qquad(p-1,p) \mapsto\begin{cases}
  (((-n+p)/2)(n-1),n+1) &\text{if\quad} p\text{\quad odd}\\
  (((-n+p+1)/2)(n-1)+2,0)  &\text{if\quad} p>0,\text{\quad even}.
\end{cases}\]
\end{itemize}
\end{construction}

As before, we then state the following result.

\begin{theorem}\label{thm:general}
Let $\D$ be a stable derivator. The following diagram commutes up to a natural equivalence and all the arrows are equivalences. 
\begin{equation}\label{commdiag}
    \begin{tikzcd}
\D^{A(n,2)} \arrow[rr, "i^n"] &                                    & \D^{A_n} \\
                  & \D^{M_n, \mathrm{ex}} \arrow[lu, "(j^n)^*"] \arrow[ru, "(i_n)^*"'] &  
\end{tikzcd}.\end{equation}
\end{theorem}
\begin{proof}
    The fact that $i^n$ is an equivalence follows from \cref{mainthm}. Next, let us look at \eqref{i^n} and notice that it is given by a composition of equivalences which can all easily be written as restrictions functors or inverse of restriction functors. By \cref{thm:AR}, we also know that $(i_n)^*$ is an equivalence: since the functors $j^n$ and $i_n$ commute with $i^n$, this implies that $(j^n)^*$ is an equivalence as well.
\end{proof}

\begin{remark}\label{symmetriesrem} One can construct symmetries of the stable derivator $\D^{M_n, \mathrm{ex}}$ by looking at the symmetries of $M_n$; they are described in \cite[Sections 4, 5 and 12]{[3]}. The most relevant for us are the \textbf{shift functor} and the \textbf{Auslander-Reiten translation}.
These two functors are, respectively, the restrictions along the following two maps \cite[(4.10),(5,9)]{[3]}: \begin{align*}\label{symmetries}
    f_n\colon &M_n\to M_n\qquad& t_n\colon &M_n \to M_n\\
    &(k,l) \mapsto (k+l,n+1-l)\qquad& &(k,l) \mapsto (k-1,l)
\end{align*}
Let $g_n$ be any finite composition of $f_n$, $t_n$, $f_n^{-1}$, $t_n^{-1}$. Then $g_nj^n$ gives us another embedding of $A(n,2)$ in $M_n$. Moreover, $(g_nj^n)^*$ is an equivalence as a functor $\D^{M_n, \mathrm{ex}}\to \D^{A(n,2)}$.
\end{remark}

\subsection{$\infty$-Dold-Kan correspondence}
 Let $\mathcal{A}$ be an abelian category. The classical Dold-Kan correspondence \cite[Theorem~1.2.3.7]{HA} asserts that the category $\mathrm{Fun}(\Delta\op,\mathcal{A})$ of simplicial objects of $\mathcal{A}$ is equivalent to the category $\mathrm{Ch}_{\ge 0}(\mathcal{A})$ of (homologically) nonnegatively graded chain complexes. By replacing $\mathcal{A}$ with a bicomplete stable $\infty$-category $\mathcal{C}$, we get an analog of the the classical Dold-Kan correspondence at level of $\infty$-categories. 

\begin{theorem}[{\cite[Theorem~1.2.4.1]{HA}}]\label{inftydold}
The $\infty$-categories \[\mathrm{Fun}(\mathrm{N}(\Delta\op),\mathcal{C})\text{ and } \mathrm{Fun}(\mathrm{N}(\Z_{\ge0}),\mathcal{C})\] are equivalent to one another. 
\end{theorem}
Here, $\mathrm{N}$ is the nerve functor and $\mathrm{Fun}(\mathrm{N}(\Z_{\ge0}),\mathcal{C})$ can be thought of as the bounded $\infty$-category of filtered objects. $\mathrm{Fun}(\mathrm{N}(\Z_{\ge0}),\mathcal{C})$ is a full subcategory of the $\infty$-category of complete filtered objects $\hat{\mathrm{Fun}}(\mathrm{N}(\Z),\mathcal{C})$ \ie  of those filtrations whose limit vanishes. It turns out that the latter is equivalent to a suitable $\infty$-category of \textbf{coherent chain complexes} that is defined as follows. Recall that $\mathrm{Ch}$ \cite[Definition~35.1]{joyal:notes} is the pointed category whose objects are $\Z \cup \{pt\}$ and whose arrows are given by \[\mathrm{Ch}(m,n)=\begin{cases}
 \{\partial_n,0\} \quad\text{    if } m=n-1,\\
    \{\id,0\} \quad\text{   if } m=n,   \\
   \{0\} \quad\quad\text{    otherwise, } 
\end{cases} \]
where, by definition, $\partial_{n-1}\partial_n=0$ and $\{pt\}$ is a zero object. Then we define the $\infty$-category  $\mathrm{Ch}(\mathcal{C})$ as the category of functors $\mathrm{Ch}\to \mathcal{C}$ preserving the zero object. We have the following result by Ariotta.
\begin{theorem}[{\cite[Theorem~4.7]{ariotta}}]\label{ari$4.7$} There exists an equivalence of stable $\infty$-categories 
\begin{equation}\label{I}
\begin{tikzcd}
{\hat{\mathrm{Fun}}(\mathrm{N}(\Z),\mathcal{C})} \arrow[rr, "\mathscr{A}", shift left=2] &  & \mathrm{Ch}(\mathcal{C}). \arrow[ll, "\mathscr{I}", shift left=2]
\end{tikzcd}
\end{equation}
\end{theorem}

\noindent The equivalence stated in the main Theorem \ref{mainthm} turns out to be the same as the equivalence in Theorem \ref{ari$4.7$} when we pass to the associated derivators and we restrict to bounded coherent chain complexes. In the following subsection, we explain how these results are related.

 \subsection{$\infty$-Dold-Kan correspondence via \cref{mainthm}}\label{dkproof}
 We describe how the modification giving the equivalence between \cref{mainthm} and \cref{ari$4.7$} is defined on objects. Let  \[\mathrm{Fun}^{[0,n]}(\mathrm{N}(\Z) ,\mathcal{C})\]  be the full subcategory of $\hat{\mathrm{Fun}}(\mathrm{N}(\Z),\mathcal{C})$ such that the images of the arrows $i-1 \to i$ are isomorphisms for all $i>n$, and such that the images of the objects $i$ are zero objects for $i<0$. We consider also the full subcategory of chain complexes supported on a given interval \[\mathrm{Ch}^{[-n,0]}(\mathcal{C})\] of $\mathrm{Ch}(\mathcal{C})$. By \cite[Remark~3.24]{ariotta}, the essential image of the equivalence (\ref{I}), restricted to $\mathrm{Fun}^{[0,n-1]}(\mathrm{N}(\Z) ,\mathcal{C})$, is $\mathrm{Ch}^{[-n+1,0]}(\mathcal{C})$. Note that here the indexing for filtrations differs from the one in \cite{ariotta}, which gives opposite signs. Then (\ref{I}) induces an equivalence between the full subcategories
\begin{equation}
\begin{tikzcd}
{\mathrm{Fun}^{[0,n-1]}(\mathrm{N}(\Z) ,\mathcal{C})} \arrow[rr, "\mathscr{A}", shift left=2] &  & \mathrm{Ch}^{[-n+1,0]}(\mathcal{C}). \arrow[ll, "\mathscr{I}", shift left=2]
\end{tikzcd}
\end{equation}

\medskip

  By \cref{exder}, the homotopy category $ \Ho(\mathrm{Fun}(\mathrm{N}(\Z) ,\mathcal{C}))$ is the homotopy derivator of $\mathcal{C}$ evaluated in $\Z$. Considering the full inclusion \[k_n: [n-1] \to \Z,\] 
 we have that the restriction $k_n^*$ gives an equivalence \[\Ho(\mathrm{Fun}^{[0,n-1]}(\mathrm{N}(\Z) ,\mathcal{C})) \stackrel{k_n^*}{\longrightarrow} \Ho_{\mathcal{C}}(A_n).\] Indeed,  $\Ho(\mathrm{Fun}^{[0,n-1]}(\mathrm{N}(\Z) ,\mathcal{C}))$ is the essential image of $(k_n)_!$. Moreover, recalling \eqref{A(n,2)},
we have a map 
\begin{align*}\label{symmetries}
    u_n\colon \tilde{A}(n,2)&\to \mathrm{Ch};\\
    (i,i+1) &\mapsto \{pt\};\\
    (0,0) &\mapsto -n+1;\\
    (n-2,n-2) &\mapsto 0;\\
    (i+1,i) &\mapsto -n+i+2;
\end{align*} whose restriction gives a functor
\[\Ho(\mathrm{Ch}^{[-n+1,0]}(\mathcal{C}))\stackrel{u^*}{\longrightarrow} \Ho_{\mathcal{C}}(A(n,2)).\]

\begin{proposition}\label{relation} If we restrict to  bounded chain complexes and bounded filtrations, then
    \cref{mainthm} is the derivator-theoretical version of \cref{ari$4.7$}. Namely, the following is a commutative diagram where all the functors are equivalences and the lower arrow is a suitable equivalence constructed as in \cref{symmetriesrem}.
\begin{equation}
\begin{tikzcd}
{\Ho(\mathrm{Fun}^{[0,n-1]}(\mathrm{N}(\Z) ,\mathcal{C}))} \arrow[rrrr, "\mathscr{A}"] \arrow[dd, "k_n^*"'] &  &  &  & {\Ho(\mathrm{Ch}^{[-n+1,0]}(\mathcal{C}))} \arrow[dd, "u^*_n"] \\
                                                                                                            &  &  &  &                                                               \\
\Ho_{\mathcal{C}}(A_n) \arrow[rrrr]                                                                         &  &  &  & {\Ho_{\mathcal{C}}(A(n,2))}                                  
\end{tikzcd}.
\end{equation}
\end{proposition}
\begin{proof}
The fact that $k_n^*$ is an equivalence was discussed above. Since the passage to the homotopy category preserves equivalences of $\infty$-categories, $\mathscr{A}$ is an equivalence and it is induced by \cref{ari$4.7$}. 

Even if we would expect the bottom arrow in the commutative diagram to be the equivalence in \cref{mainthm}, this is not sufficient: we need a slightly more complicated composition of functors which factors through the coherent Auslander-Reiten quiver (\cref{reprth}). Let us describe why.
 
In \cite[Remark~3.24]{ariotta} Ariotta draws a diagram which describes the equivalence in his Theorem; look also at \cite[Remark~4.9]{ariotta}. Let us notice that such a diagram has the shape $M_n$ and that there are the same conditions of zero objects, commutativity and bicartesian squares which characterize the homotopy derivator $\Ho_{\mathcal{C}}^{M_n, ex}$ (\cref{constr:M_n}). Indeed 
\[\Ho_{\mathcal{C}}(M_n)=\Ho(\mathrm{Fun}(\mathrm{N}(M_n) ,\mathcal{C})).\]

Moreover, thanks to \cite[Example~4.3.2.4]{HTT} and the definition of (co)cartesian squares in \cite[Section~4.4.2]{HTT}, we can define bicartesian squares through Kan extensions and then, the coherent diagram we get by looking at $\Ho_{\mathcal{C}}(M_n)^{ex}$, is exactly Ariotta's diagram in \cite[Remark~3.24]{ariotta}.

In \cref{reprth}, we saw that the equivalence of \cref{mainthm} factors through $\D^{M_n, ex}$ (\cref{thm:general}). In particular, if we take $\D$ to be the homotopy derivator $\Ho_{\mathcal{C}}$, we also have the following commutative diagram
\begin{equation}\label{Ho}
\begin{tikzcd}
\Ho_{\mathcal{C}}(A_n) \arrow[rd, "((i_n)^*)^{-1}"'] \arrow[rr, "G^n"] &                                                    & {\Ho_{\mathcal{C}}(A(n,2))} \\
                                                                       & \Ho_{\mathcal{C}}(M_n)^{ex} \arrow[ru, "(j^n)^*"'] &                            
\end{tikzcd}.
\end{equation}
The embedding $i_n$ describes precisely the same objects as in \cite[Remark~3.24]{ariotta}. But, if we substitute $G^n$ with $u_n^*\mathscr{A}$ in (\ref{Ho}), we don't get a commutative diagram anymore. This is because $u_n^*\mathscr{A}$ differs from $G^n$ by an autoequivalence of $\Ho_{\mathcal{C}}(M_n)^{ex}$ which we define as follows. 
\noindent The map $j^n$, defined in \cref{cons:jn}, when $n=4$, gives the position of the bold objects in the diagram $M_4$, \cref{$n=4$}. From this embedding we get Ariotta's complex by composing with an autoequivalence $(b_n)^*$ which is a composition of shift functors $(f_n)^*$ and Auslander-Reiten translations $(t_n)^*$. We will prove this in \cref{egs:doldkan}, for $n=3$; for general $n$, the argument is analogous. In conclusion, the equivalence between \cref{mainthm} and \cref{ari$4.7$} is given by the following commutative diagram:

\begin{equation}
\begin{tikzcd}
{\Ho(\mathrm{Fun}^{[0,n-1]}(\mathrm{N}(\Z) ,\mathcal{C}))} \arrow[rrrr, "\mathscr{A}"] \arrow[dd, "k_n^*"'] &  &                                                 &  & {\Ho(\mathrm{Ch}^{[-n+1,0]}(\mathcal{C}))} \arrow[dd, "u_n^*"] \\
                                                                                                            &  &                                                 &  &                                                               \\
\Ho_{\mathcal{C}}(A_n) \arrow[rr, "(i_n^*)^{-1}"]                                                           &  & \Ho_{\mathcal{C}}(M_n)^{ex} \arrow[rr, "b_n^*"] &  & {\Ho_{\mathcal{C}}(A(n,2))}                                  
\end{tikzcd}
\end{equation}

\noindent where also $u_n^*$ is an equivalence because all the other functors are and the diagram commutes. \end{proof}
\begin{example}\label{egs:doldkan}
The diagram below is $M_3$ and it allows us to compare \cref{mainthm} and \cref{ari$4.7$} through \cref{thm:AR}. 
\begin{equation}\label{manifestation}
\adjustbox{scale=0.6}{
\begin{tikzcd}
                    &                               &                               &                               & {\textbf{(-2,4)}} \arrow[d]            & {(-1,4)} \arrow[d]            & {(0,4)} \arrow[d]           & {(1,4)} \arrow[d]            & {(2,4)} \arrow[d]            & {(3,4)} \arrow[d]  & {(4,4)} \\
                    &                               &                               & {\textbf{(-2,3)}} \arrow[ru] \arrow[d] & {(-1,3)} \arrow[ru] \arrow[d] & {\textbf{\underline{(0,3)}},\mathscr{I}C^0} \arrow[ru] \arrow[d] & {(1,3)} \arrow[ru] \arrow[d] & {c^1} \arrow[ru] \arrow[d] & {(3,3)} \arrow[ru] \arrow[d] & {(4,3)} \arrow[ru] &         \\
                    & \cdots                        & {(-3,2)} \arrow[ru] \arrow[d] & {(-1,2)} \arrow[ru] \arrow[d] & {\underline{(0,2)},\mathscr{I}C^1} \arrow[ru] \arrow[d] & {(1,2)} \arrow[ru] \arrow[d]  & {(2,2)} \arrow[ru] \arrow[d] & {(3,2)} \arrow[ru] \arrow[d] & {(4,2)} \arrow[ru]           & \cdots             &         \\
                    & {(-2,1)} \arrow[ru] \arrow[d] & {(-1,1)} \arrow[ru] \arrow[d] & {\underline{\textbf{(0,1)}},\mathscr{I}C^2} \arrow[ru] \arrow[d] & {(1,1)} \arrow[ru] \arrow[d]  & {c^2} \arrow[ru] \arrow[d]  & {(3,1)} \arrow[ru] \arrow[d] & {c^0} \arrow[ru]           &                              &                    &         \\
{(-2,0)} \arrow[ru] & {(-1,0)} \arrow[ru]           & {(0,0)} \arrow[ru]           & {(1,0)} \arrow[ru]            & {(2,0)} \arrow[ru]            & {(3,0)} \arrow[ru]            & {(4,0)} \arrow[ru]           &                              &                              &                    &        
\end{tikzcd}}    
\end{equation}

\smallskip
\noindent This diagram comes from Construction \ref{constr:M_n} behind Theorem \ref{thm:AR}, where the functor \[i_3:A_3 \to M_3\] embeds $A_3$ in the underlined coordinates $(0,1),(0,2),(0,3)$. The bold coordinates are the ones describing the embedding \[j^3:A(3,2) \to M_3\] in Theorem \ref{thm:general}. The coordinates denoted by $\mathscr{I}C^0$, $\mathscr{I}C^1$, $\mathscr{I}C^2$ and $c^0$, $c^1$, $c^2$ describe the link with \cite[Remark~3.24]{ariotta} and so with the equivalence (\ref{I}). The autoequivalence we are searching for is the one which allows to pass from the complex given by the bold objects $(-2,3), (0,1), (0,3)$ to the one given by the objects $c^0, c^1, c^2$. Namely, let $C$ be a complex in $\mathrm{Ch}_{[2,0]}(\mathcal{C})$
\[C:\quad \cdots \to 0 \to C^2 \to C^1 \to C^0 \to 0 \to \cdots; \]
as mentioned before, the image of $C$ under the functor $\mathscr{I}$ (\ref{I}) coincides with the objects in the coordinates $(0,1),(0,2),(0,3)$, where also $A_3$ is embedded through $i_3$. Let \[b_3: \square \to M_3\] be the embedding whose image is the square with vertices $c^0,c^1,c^2,(4,0)$. Observe that $b_3^*$ has the  following form
\begin{equation}\label{b}
\begin{tikzcd}
{\D^{M_3, \mathrm{ex}}} \arrow[rr, "f_3^*(t_3^*)^{-1}"] &  & {\D^{M_3, \mathrm{ex}}} \arrow[rr, "(j^3)^*"] &  & {\D^{A(3,2)}}
\end{tikzcd},
\end{equation}
where $f_3,t_3$ were defined in \cref{symmetriesrem}. In particular, since all the maps in (\ref{b}) are equivalences, $b_3^*$ is an equivalence. This means that if we consider an object $X \in \D^{M_3, \mathrm{ex}}$, then we have the following isomorphisms
\begin{align*}
  & f_3^*(t_3^*)^{-1}(X)_{(-2,3)}\cong X_{c_2},\quad\qquad && f_3^*(t_3^*)^{-1}(X)_{(0,1)}\cong X_{c_1},\\
  & f_3^*(t_3^*)^{-1}(X)_{(0,3)} \cong X_{c_0},\quad\quad\,\, && f_3^*(t_3^*)^{-1}(X)_{(-2,4)} \cong X_{(4,0)}\cong 0.
\end{align*}

Moreover, considering the map
\[b_3^*(i_3^*)^{-1}:\D^{A_3}\to \D^{A(3,2)},\]
we get the same equivalence as Ariotta in \cite[Theorem~4.7]{ariotta}.
\end{example}

\section{Universal tilting theory}\label{sec:tilting}

The goal of this section is to show an additional link that our main \cref{mainthm} has with homotopy theory and in particular with tilting theory. More precisely, we aim to prove that the functors which give the equivalence (\ref{maineq}) can be realized as tensor products with {\em spectral bimodules.} This is a universal version of the (derived) Morita Theory developed by J. Rickard \cite{rickard89}. In particular, we have the following well known result. 
\begin{theorem}[{\cite[Theorem~6.4]{rickard89}}]
Let $k$ be a commutative ring and $A,B$ $k$-algebras which are flat as modules over $k$. The following are equivalent.
\begin{enumerate}
    \item There is a $k$-linear triangle equivalence $\Phi: \mathrm{D}(\Mod A) \to \mathrm{D}(\Mod B)$.
    \item There is a complex of $A$-$B$-modules $X$ such that the total left derived functor \[ -\otimes^\mL_A X: \mathrm{D}(\Mod A) \to \mathrm{D}(\Mod B) \]
    is an equivalence.
\end{enumerate}
\end{theorem}

As illustrated in Example \ref{egs:stable}, derived categories are precisely the underlying categories of a specific stable derivator. It is thus natural to look for a generalization of Rickard's result at level of derivators. To achieve it, let us recall that every stable derivator is canonically a closed module over the derivator of spectra $\cSp$ \cite[Appendix~A.3]{cisinskitabuada11}. Thus if $\D$ is a stable derivator, there is a canonical action
\[\otimes\colon\cSp \times \D \longrightarrow \D\]
which, for every $A,B,C \in \cCat$, allows us to define the so called \textbf{canceling tensor product} \cite[Section~5]{gps:additivity} \begin{align*}
\otimes_{[A]}\colon\cSp(B\times A\op) \times \D(A\times C) &\longrightarrow \D(B \times C)\\
(X,Y) &\mapsto\,\,\,\, X \otimes_{[A]}Y.
\end{align*}
Additionally, let us recall that we refer to an object of $\cSp(B \times A\op)$ as \textbf{spectral bimodule}.

\medskip

Our goal in the present section is to apply a derivator enhancement of Rickard's result to the functor \[G^n:\D^{A_n}\to \D^{A(n,2)}\] in \cref{mainthm}. In particular, we aim to show that $G^n$ is equivalent to a functor whose components are canceling tensor products. In other words, for every $B \in \cCat$, we aim to exhibit an equivalence between $\D(A_n \times B)$ and $\D(A(n,2) \times B)$ in the form of a canceling tensor product. Recall that an object in $\D(A(n,2) \times B)$ is an object in $\D(\tilde{A}(n,2) \times B)$ subject to some vanishing conditions (cf. Definition \ref{def:A(n,2)}). It is hence natural to look for a spectral bimodule in $\cSp(A(n,2) \times A_n\op)$.
As proved and defined in \cite[Theorem~5.9]{gps:additivity}, the unit of the canceling tensor product is given by the \textbf{identity profunctor}
\[\mathbb{I}_{A_n} \in \cSp(A_n \times A_n\op) \cong \cSp^{A_n}(A_n\op).\]
Applying Theorem \ref{mainthm} to $\cSp$ yields the following equivalence \begin{equation}\label{speq2}
    \cSp^{A_n} \stackrel{\sim}{\longrightarrow} \cSp^{A(n,2)}.\end{equation}
This allows us to define a particular spectral bimodule 
\[
T_n \in \cSp(A(n,2) \times A_n\op)\]
as being the image under the equivalence (\ref{speq2}) of the identity profunctor $\mathbb{I}_{A_n}$:
\begin{align*}
 \cSp^{A_n}(A_n\op) &\stackrel{\sim}{\longrightarrow} \cSp^{A(n,2)}(A_n\op)\cong \cSp(A(n,2) \times A_n\op)\\
 \mathbb{I}_{A_n} &\mapsto \quad T_n.
\end{align*}
For every small category $B$, we then define an action of the bimodule $T_n$ on $\D(A_n \times B)$ via the canceling tensor product
\[ \otimes_{[A_n]}\colon \cSp(A(n,2) \times A_n\op) \times \D(A_n \times B) \longrightarrow \D(A(n,2) \times B).\]
Namely, we define the functor
\begin{align*}
T_n\otimes_{[A_n]}-\colon \D(A_n \times B) &\longrightarrow \D(A(n,2) \times B)\\
X &\mapsto T_n\otimes_{[A_n]}X.
\end{align*}

What is left to discuss is why the functor $G^n$ in \eqref{maineq} is isomorphic to $T_n\otimes_{[A_n]}-$. For this purpose, let us recall the following definition.
\begin{defn}[{\cite[Definition~8.1]{[3]}}] Let $\D$ be a stable derivator and let $A,B\in\cCat$. A morphism $\D^A\to\D^B$ is \textbf{left admissible} if it can be written as a composition of 
\begin{itemize}
\item[(LA1)] restriction morphisms $u^\ast\colon\D^{B'}\to\D^{A'}$,
\item[(LA2)] left Kan extensions $u_!\colon\D^{A'}\to\D^{B'}$, 
\item[(LA3)] right Kan extensions $u_\ast\colon\D^{A'}\to\D^{B'}$ along fully faithful functors which amount precisely to adding a cartesian square or right Kan extensions along countable compositions of such functors, and
\item[(LA4)] right extensions by zero $u_\ast\colon\D^{A'}\to\D^{B'}$ for sieves $u\colon A'\to B'$.
\end{itemize}
Dually, we define a \textbf{right admissible} morphism. \end{defn}

By the construction in the proof of \cref{mainthm}, the functor $G^n$ is left admissible. This allows us to apply the following result, stating that every left admissible morphism is a canceling tensor product. 

\begin{theorem}[{\cite[Theorem~8.5]{[3]}}]\label{specact}
Let $\D$ be a stable derivator and let $F\colon\D^A\to\D^B$ be a morphism.
If $F$ is left admissible then there is a bimodule $M\in\cSp(B\times A\op)$ and a natural isomorphism \[F\cong M\otimes_{[A]}-\colon\D^A\to\D^B.\]
\end{theorem}
\noindent The proof of \cite[Theorem~8.5]{[3]} shows that in our case the module $M$ is given by $T_n$ and then we can conclude with the following proposition

\begin{proposition}\label{tensor} The following equivalence of functors holds
    \[ G^n \cong T_n\otimes_{[A_n]}-.\]
\end{proposition}

\section*{Declarations}
This research is mainly supported by GA{\v C}R project 20-13778S and partially by the Charles University project PRIMUS/21/SCI/014.
Data sharing not applicable to this article as no datasets were generated or analyzed during the current study. Ethical approval, competing interests and authors' contributions are not applicable. 

\end{document}